\documentclass[12pt]{article}
\usepackage{graphicx,enumerate,amsthm,authblk}
\usepackage{amsmath,amssymb,multirow,array,float,bm,bbm}
\usepackage[colorlinks,citecolor=blue,urlcolor=blue,linkcolor=red]{hyperref}
\usepackage[capitalize]{cleveref}
\usepackage{caption}
\usepackage{subcaption}

\newtheorem{definition}{Definition}
\newtheorem{proposition}{Proposition}
\newtheorem{assumption}{Assumption}
\newtheorem{lemma}{Lemma}
\newtheorem{theorem}{Theorem}
\newtheorem{corollary}{Corollary}
\newtheorem{remark}{Remark}
\newtheorem{example}{Example}
\newcommand{\bigO}{\mathcal{O}}
\newcommand{\dee}{\mathrm{d}}
\newcommand{\eps}{\varepsilon}
\renewcommand{\Pr}{\mathbb{P}}
\def\EE{\mathbb{E}}
\def\bR{\mathbb{R}}

\begin{document}

	\title{\bf\large Estimation and Inference of Time-Varying Auto-Covariance under Complex Trend: A Difference-based Approach
	\footnotetext{Email: cui@utstat.toronto.edu; mlevins@purdue.edu; zhou@utstat.toronto.edu.}}
	
\author[1]{Yan Cui} 
\author[2]{Michael Levine}
\author[1]{Zhou Zhou} 
\affil[1]{\small Department of Statistical Sciences, University of Toronto, Canada}
\affil[2]{\small Department of Statistics, Purdue University, USA}
\maketitle

\begin{abstract}
We propose a difference-based nonparametric methodology for the estimation and inference of the time-varying auto-covariance functions of a locally stationary time series when it is contaminated by a complex trend with both abrupt and smooth changes. Simultaneous confidence bands (SCB) with asymptotically correct coverage probabilities are constructed for the auto-covariance functions under complex trend. A simulation-assisted bootstrapping method is proposed for the practical construction of the SCB. Detailed simulation and a real data example round out our presentation. \\
\textit{Key words}: Change points; Gaussian approximation; Local stationarity; Simultaneous confidence bands.
\end{abstract}

\section{Introduction}

Our discussion begins with a heteroscedastic nonparametric regression model 
\begin{equation}\label{model1}
Y_{i}=\mu(t_{i})+\sqrt{V(t_{i})}\eps_{i}\,, i=1,\ldots,n,
\end{equation}
where $Y_{i}$ are the observations, $\mu$ is an unknown mean function, $t_{i}=i/n$ are the design  points, $i=1,\ldots,n,~\eps_{i}$ are the errors with mean zero and variance $1,$ and $V$ is the variance function. Historically, it has been assumed that the errors $\eps_{i}$ are independent. Variance estimation in regression models with the unknown mean has traditionally been a rather important problem. Accurate variance estimation is required for the purpose of, for example, construction of confidence bands for the mean function, testing the goodness of fit of a model, and also in order to choose the amount of smoothing needed to estimate the mean function; see e.g. \cite{rice1984bandwidth}, \cite{eubank1990testing}, \cite{gasser1991flexible}, and \cite{hardle1997local}.  An extensive survey of the difference sequence approach to estimate the variance in the nonparametric regression setting when the variance function is only a constant can be found in \cite{dai2017choice}. 

The situation when the variance is not constant is more complicated. One of the first attempts to estimate the variance function in a regression model was made in \cite{muller1987estimation} who proposed the basic idea of kernel smoothing of squared differences of observations. This idea has been further developed in \cite{muller1993variance}. \cite{brown2007variance} introduced a class of difference-based local polynomial regression-based estimators of the variance function $V(x)$ and obtained optimal convergence rates for this class of estimators that are uniform over broad functional classes. \cite{wang2008effect} obtained the minimax rate of convergence for estimators of the variance function in the model \eqref{model1} and characterized the effect of not knowing the mean function on the estimation of variance function in detail. Similar approach was used to construct a class of difference-based estimators in \cite{cai2009variance} when the covariate $X\in \bR^{d}$ for $d>1.$

All of the above mentioned papers only considered the case where the data are independent. However, difference-based methods have also been used to estimate variance and/or autocovariance in nonparametric regression where the errors are generated by a stationary process. The pioneering approach here was probably that of \cite{hans1988detecting} who proposed estimators based on the first-order differences to estimate (invertible) linear transformations of the variance-covariance matrix of stationary $m$-dependent errors. Here, by $m$-dependent errors we mean the errors generated by a stationary process whose autocovariance is equal to zero for any lag greater than some $m>0.$  \cite{herrmann1992choice} suggested second order differences to estimate the zero frequency of the spectral density of stationary processes with short-range dependence. In the case of autoregressive errors,  \cite{hall2003using} proposed $\sqrt{n}$-root consistent and, under the assumption of normality of errors, efficient estimators of the autocovariance that are also based on differences of observations. Under certain mixing conditions, \cite{park2006simple} proposed estimating the autocovariance function by applying difference-based estimators of the first order to the residuals of a kernel-based fit of the signal. \cite{zhou2015optimal} provided an optimal difference-based estimator of the variance for smooth nonparametric regression when the errors are correlated. Finally, the closest to us in spirit is, probably, \cite{tecuapetla2017autocovariance} that proposed a class of difference-based estimators for the auto-covariance in nonparametric regression when the signal is discontinuous and the errors form a stationary $m$-dependent sequence.  To the best of our knowledge, the problem of auto-covariance estimation in a nonparametric regression where the errors form a non-stationary sequence while the signal is discontinuous has not been considered before.  

The purpose of this article is to estimate and make inference of the time-varying covariance structure of a locally stationary time series when it is contaminated by a complex trend function with both smooth and abrupt changes. Here local stationarity refers to the slowly or smoothly evolving data generating mechanism of a temporal system (\cite{Dahlhaus1997}, \cite{Nason2000}, \cite{Zhou2009}). In time series analysis, the estimation and modelling of the auto-covariance structure is of fundamental importance in, for example, the optimal forecasting of the series (\cite{Brockwell2016}), the efficient estimation of time series regression models (\cite{Hamilton1994}) and the inference of time series regression parameters (\cite{Brockwell2016}). When the trend function is discontinuous, removing the trend from the time series and then estimating the auto-covariances from the residuals is not a good idea since it is very difficult to estimate the trend function near the points of discontinuity accurately. In this case, the aforementioned difference-based methods offer a good alternative. In this paper, we adopt a difference-based local linear regression method for the aforementioned time-varying auto-covariance estimation problem. The method can be viewed as a nonparametric and non-stationary extension to \cite{tecuapetla2017autocovariance}. It is shown that the uniform convergence rate of auto-covariance function estimation for the difference-based method under complex trend is the same as that of auto-covariance function estimation of a zero-mean time series when the number of points of discontinuity as well as the jump sizes diverge to infinity at a sufficiently slow rate. Therefore, asymptotically, the accuracy of auto-covariance function estimation will not be affected by the complex trend when the difference-based nonparametric method is used.

Making inference of the auto-covariance functions is an important task in practice as practitioners and researchers frequently test whether certain parametric or semi-parametric models are adequate to characterize the time series covariance structure. For instance, one may be interested in testing whether the auto-covariance functions are constant over time so that a weakly stationary time series model is sufficient to forecast the future observations. There is a rich statistical literature on the inference of auto-covariance structure of locally stationary time series, particularly on the testing of weak stationarity of such series. See for instance \cite{Papatoditis2010}, \cite{Dwivedi2011}, \cite{dette2011optimal}, \cite{Nason2013}, \cite{Dette2019} and \cite{Ding2019}. To our knowledge, only constant or smoothly time-varying trend were considered in the aforementioned literature of covariance inference. In this paper, simultaneous confidence bands (SCB) with asymptotically correct coverage probabilities are constructed for the time-varying auto-covariance functions when estimated by the difference-based local linear method. The SCB serves as an asymptotically correct tool for various hypothesis testing problems of the auto-covariance structure under discontinuous mean functions. A general way to perform such hypothesis tests is to estimate the auto-covariance functions under the parametric or semi-parametric null hypothesis and then check whether the fitted functions can be fully embedded into the SCB. As the auto-covariance functions can be estimated with faster convergence rates under the parametric or semi-parametric null hypothesis, the aforementioned way to perform the test achieves correct Type-I error rate asymptotically. The tests are of asymptotic power 1 for local alternatives whose uniform distances from the null are of the order greater than that of the width of the SCB, see Theorem 2 in \cite{Zhou2010} for instance. We also propose a simulation-assisted bootstrapping method for the practical construction of the SCB.  

The paper is organized as follows. In \cref{sec2}, we introduce the model formulation and some assumptions on $x_i$ and $\eps_i^k$. \cref{sec3} presents the asymptotic theory for local estimate $\beta_k(\cdot)$. Practical implementation including a suitable difference lag and tuning parameters selection procedure, estimation of covariance matrices as well as an assisted bootstrapping method are discussed in \cref{sec4}. In \cref{sec_sim}, we conduct some simulation experiments on the performance of our SCBs. A real data application is provided in \cref{sec6}. The proofs of the main results are deferred to the Appendix.  

	\section{Model formulation}\label{sec2}
	Consider model:
	\begin{equation}\label{e1}
	y_{i,n}=\mu_{i,n}+x_{i,n}
	\end{equation}
	where $\mu_{i,n}:=\mu(t_i)$ is a mean function or signal with unknown change points, $y_{i,n}=y(t_i)$ and $x_{i,n}=x(t_i):=G(\frac{i}{n},\mathcal{F}_i)$ is a zero-mean locally stationary process with $t_i=i/n,~i=1,...,n$. \cref{e1} covers a wide range of nonstationary linear and nonlinear processes, see \cite{Zhou2009} for more discussion. We shall omit
	the subscript $n$ in the sequel if no confusion arises. Let $\zeta_i,~i\in\mathbb{Z}$, be  independent identically distributed (i.i.d.) random variables, and define $\mathcal{F}_i=(...,\zeta_{i-1},\zeta_i)$. Then, the process $\{x_i\}$ can be written as  $$x_i=G(t_i,\mathcal{F}_{i}),$$
	where $G(\cdot,\cdot)$ is a measurable function such that $G(t,\mathcal{F}_i)$ is well defined for all $t\in[0,1]$. In this paper, we focus on the case that there exists $0=a_0<a_1<\cdots<a_d<a_{d+1}=1$ such that
	\begin{equation*}
	\mu(t)=\sum_{j=0}^d\mu_j(t)\mathbbm{1}_{\{a_j\le t<a_{j+1}\}},
	\end{equation*}
	where $\mu_j(t)$ is a Lipschitz continuous function over $[a_j,a_{j+1})$ and $d$ is the total number of change points. Till the end of this paper, we will always assume $d=d_n=\bigO(n^\alpha)$ and the maximal jump size $\Delta_n=\max_{1\le j\le d}|\mu_j(a_j)-\mu_{j}(a_j^-)|=\bigO(n^\beta)$ with $0\le \alpha,~\beta<1$.
	
	To estimate the second order structure of the process \cref{e1}, we introduce the approach based on the difference sequence of a finite order applied to the observations $y_i$. Assuming that the number of observations is $n+k$, this difference-based  covariance estimation approach would define simple squared differences of the observations, i.e., $\rho_i^k:=\rho_k(t_{i-k})=(y_i-y_{i-k})^2,~i=k+1,k+2,...,k+n$. Notice that for any fixed $t$, $\{\xi_i(t):=G(t,\mathcal{F}_i)\}_{i\in\mathbb{Z}}$ is a stationary process. For convenience, let us denote $s_i=t_{i-k}=\frac{i-k}{n}.$ Then, $\gamma_k(s_i)$ is the $k$th order autocovariance function of the process $\{x_i\}$ at the fixed time $s_i;$ in other words, 
	$\gamma_k(s_i):={\rm Cov}(x_i,x_{i-k}),~k=0,1,...$. If $k=0$, then $\gamma_0(s_i)$ turns out to be the variance of $x_i$. 
	
	We first introduce some notation that will be used throughout this paper. For any vector $v=(v_1,v_2,...,v_p)\in \mathbb{R}^p$, we let $|v|=(\sum_{i=1}^pv_i^2)^{1/2}$. For any random vector $V$, write $V\in\mathcal{L}^p~(p>0)$ if $\Vert V\Vert_p:=\EE(|V|^p)^{1/p}<\infty$. Denote $C^p([0,1])$ as the function space on [0,1] of functions that have continuous first $p$ derivatives with integer $p>0$. Now, we need the following definition and assumptions:
	\begin{definition}[Physical dependence measure]
		Let $(\zeta_j')_{j\in\mathbb{Z}}$ be an i.i.d. copy of $(\zeta_j)_{j\in\mathbb{Z}}$. Then, for any $j\geq 0$, we denote $\mathcal{F}_j'=(\mathcal{F}_{-1},\zeta_0',\zeta_1,
		...,\zeta_j).$ The physical dependence measure for a stochastic system $L(t,\mathcal{F}_j)$ is defined as
		\begin{equation}\label{e2}
		\delta_q(L,j)=\sup_{t\in[0,1]}\Vert L(t,\mathcal{F}_j)-L(t,\mathcal{F}_j')\Vert_q.
		\end{equation}
	\end{definition}
	If $j<0$, let $\delta_q(L,j)=0$. Thus, $\delta_q(L,j)$ measures the dependence of the output $L(t,\mathcal{F}_j)$ on the single input $\zeta_0$; see \cite{Wu2005} for more details.

    \begin{assumption}\label{a1} 
    	$\delta_8(G,l)=\bigO(l^{-2})$ for $l\ge 1$.
    \end{assumption}

    \begin{assumption}[Stochastic Lipschitz continuity]\label{a2}
    There exists a constant $C>0$, such that $\Vert G(t,\mathcal{F}_i)-G(s,\mathcal{F}_i)\Vert_4 \leq C|t-s|$ holds for all $t,~s\in[0,1]$ and  $\sup_{t\in[0,1]}\Vert G(t,\mathcal{F}_i)\Vert_8 < \infty$.
    \end{assumption}	
	
	\cref{a1} shows that the dependence measure of time series $\{x_i\}$ decays at a polynomial rate, thus indicating short-range dependence. \cref{a2} means that $G$ changes smoothly over time and ensures local stationarity. Here, we show some examples of the locally stationary linear and nonlinear time series that satisfy these assumptions.
	
	\begin{example}[Nonstationary linear processes]
		Let $\zeta_i$ be i.i.d. random variables with $\zeta_i\in \mathcal{L}^q,~q\ge 1$; let $a_j(\cdot),~j=0,1,...,$ be $\mathcal{C}^1([0,1])$ functions such that
		\begin{equation}\label{loc.stat.linear}
		G(t,\mathcal{F}_i)=\sum_{j=0}^\infty a_j(t)\zeta_{i-j}
		\end{equation}
		is well defined for all $t\in [0,1]$. Clearly by \cite[Proposition 2]{Zhou2009}, we know that \cref{a1} will be satisfied if $\sup_{t\in[0,1]}\{|a_l(t)|^2\}=\bigO(l^{-2})$. Furthermore, if $\sum_{j=0}^\infty\{\sup_{t\in[0,1]}|a_j'(t)|^2\}<\infty$, the
		stochastic Lipschitz continuity condition in \cref{a2} also holds true.
	\end{example}
	
	\begin{example}[Nonstationary nonlinear processes]
		Let $\zeta_i$ be i.i.d. random variables and consider the nonlinear time series framework
		\begin{equation}\label{nonlinear}
		\xi_i(t)=R(t,\xi_{i-1}(t),\zeta_i),
		\end{equation}
		 where $R$ is a measurable function and $t\in [0,1]$. This form has been introduced by \cite{Zhou2009} and \cite{Zhou2013}. Suppose that for some $x_0$, we have $\sup_{t\in[0,1]}\Vert R(t,x_0,\zeta_i)\Vert_q<\infty$ for $q>0$. Denote $$\chi:=\sup_{t\in[0,1]}L(t),~\text{where}~L(t)=
		\sup_{x\neq y}\frac{\Vert R(t,x,\zeta_0)-
			R(t,y,\zeta_0)\Vert_q}{|x-y|}.$$ 
		It is known from \cite[Theorem 6]{Zhou2009} that if $\chi<1$, then \cref{nonlinear} admits a unique locally stationary solution with $\xi_i(t)=G(t,\mathcal{F}_i)$ and the physical dependence
		measure satisfies that $\delta_q(G,j)\le C\chi^j$, which shows geometric moment contraction. Hence, the temporal dependence with exponentially decay indicates \cref{a1} holds with $q=8$. Further by \cite[Proposition 4]{Zhou2009}, we conclude that \cref{a2} holds for $q=4$ if 
		$$\sup_{t\in [0,1]}\Vert M(G(t,\mathcal{F}_0))\Vert_q<\infty,
		~\text{where}~M(x)=\sup_{0\le t<s\le 1}\frac{\Vert R(t,x,\zeta_0)-
			R(s,x,\zeta_0)\Vert_q}{|t-s|}.$$ 
	\end{example}

Due to the local stationarity of the process $\{x_i\}$, we have the following lemma which shows that, under mild assumptions, the auto-covariance of $\{x_i\}$ also exhibits polynomial decay.
	
	\begin{lemma}\label{lemma1}
		Suppose Assumptions \ref{a1} and \ref{a2} hold, then we have $\gamma_k(t)=\bigO(k^{-2})$ for $k\ge 1$.
	\end{lemma}
	
With the above result, we can choose $h$ large enough such that $\gamma_{k}(t)\approx0$ for $k\ge h$. Next we focus on the difference series $\rho_i^k$ for $k=1,...,h$ and we always assume $h=h_n=\bigO(n^{1/4}\log n)$. By \cref{e1}, we know that 
\begin{align}\label{e3}
\rho_i^k&=(x_i-x_{i-k})^2+(\mu_i-\mu_{i-k})^2+2(x_i-x_{i-k})(\mu_i-\mu_{i-k}) \notag\\
:&=\alpha_i^k+\lambda_i^k+\theta_i^k.
\end{align}
Recall $s_i=(i-k)/n$ for $i=k+1,...,.k+n$ and notice that $\alpha_i^k:=\alpha_k(s_i)=(x_i-x_{i-k})^2$ is the squared difference of two locally stationary processes. Therefore, it is also a locally stationary process. As a result, we can define
\begin{equation}\label{e4}
\alpha_i^k=(x_i-x_{i-k})^2=\beta_k(s_i)+\eps_i^k,~k=1,...,h,
\end{equation}
where $\beta_k(\cdot)$ is the unknown trend function and $\eps_i^k:=\eps_k(s_i)$ is a zero-mean process. Then $\eps_i^k$ can be written as
\begin{equation}\label{e5}
\eps_i^k=H_k(s_i,\mathcal{F}_{i}),
\end{equation}
where $H_k$ is a measurable function similar to $G$. With \cref{e4}, if the trend function is smooth, one can easily obtain the estimator of $\beta_k(\cdot)$. Now, we introduce the following conditions. 

\begin{assumption}\label{a3}
For each $k=0,...,h-1$, we assume that the $k$th order autocovariance function $\gamma_k(t) \in \mathcal{C}^3([0,1])$.
\end{assumption}

\begin{assumption}\label{a4}
	The smallest eigenvalue of $\sigma_k(t)$ is bounded away from 0 on $[0,1]$ for $k=1,...,h$, where
	\begin{equation}\label{e12}
	\sigma_k(t)=\left\{\sum_{j=-\infty}^{\infty}{\rm Cov}(H_k(t,\mathcal{F}_0),H_k(t,\mathcal{F}_j))\right\}^{1/2},
	\end{equation}
	 and $\sigma_k^2(t)$ represents the long-run variance of $\eps_k(t)$ for each fixed $t\in[0,1]$.
\end{assumption}

\begin{assumption}\label{a5}
	A kernel $K(\cdot)$ is a symmetric proper density function with the compact support $[-1,1]$.
\end{assumption}

\cref{a3} guarantees that the trend function $\beta_k(\cdot)$ changes smoothly for each $k=1,...,h$ and is three-times continuously differentiable over $[0,1]$. \cref{a4} prevents the asymptotic multicollinearity of regressors. \cref{a5} allows popular kernel functions such as Epanechnikov kernel. Now substituting \cref{e4} to \cref{e3}, we have
\begin{equation}\label{e6} 
\rho_i^k=\beta_k(s_i)+\eps_i^k+\lambda_i^k+\theta_i^k.
\end{equation}

Since the length of the series $\{\rho_i^k\}_{i=k+1}^{k+n}$ is $n$, we  reset the subscript with respect to $i$ as $\{\rho_i^k\}_{i=1}^n$ and therefore the time point turns out to be $t_i=i/n$ for $i=1,...,n$. Similar notations are used for series $\{\eps_i^k\},~\{\lambda_i^k\}$ and $\{\theta_i^k\}$. By \cref{a3} and the Taylor's expansion on $\beta_k(\cdot)$, it is natural to estimate $\beta_k(t)$ using the local linear estimator as follows:
  
\begin{equation}\label{e7}
(\widehat{\beta}_{k,b}(t),\widehat{\beta}_{k,b}'(t))=\mathop{\arg\min}_{c_0,c_1\in\mathbb{R}}\left[\sum_{i=1}^{n}[\rho_i^k-c_0-c_1(t_i-t)]^2K_b
(t_i-t)\right],
\end{equation}
where $t_i=i/n$ with $i=1,...,n$ and $K_b(\cdot)=K(\cdot/b)$ is a kernel function, $b=b_n$ is the bandwidth satisfying $b\to 0$ and $nb\to\infty$. Since \cref{e7} is essentially a weighted least squares estimate, we can write the solution of \cref{e7} as
\begin{equation}\label{e8}
\widehat{\beta}_{k,b}(t)=\sum_{i=1}^{n}\omega_n^b(t,i)\rho_i^k,
\end{equation}
where $\omega_n^b(t,i)=K_b(t_i-t)\frac{S_2^b(t)-(t_i-t)S_1^b(t)}
{S_2^b(t)S_0^b(t)-[S_1^{b}(t)]^2}$ with $S_j^b(t)=\sum_{i=1}^{n}(t_i-t)^jK_b(t_i-t)$, $j=0,1,2$. The time domain of $t$ is fixed over $[0,1]$ and $\omega_n^b(t,i),~n(t-b)\leq i\leq n(t+b)$ is the weight given to each observation. 

Next, we will establish the following two lemmas that are useful in establishing asymptotic properties of proposed estimators. Their proofs are given in the Appendix. 

\begin{lemma}\label{lemma2}
	Suppose Assumptions \ref{a1}-\ref{a2} hold, then we have $\delta_4(H_k,l)=\bigO(l^{-2})$ for $0<l<k$ and $\delta_4(H_k,l)=\bigO(l^{-2})+\bigO((l-k+1)^{-2})$ for $l\ge k$, where $k=1,..,h$.
\end{lemma}

\begin{lemma}\label{lemma3}
	Suppose Assumptions \ref{a1}-\ref{a3} hold, then we have
	$\Vert H_k(t,\mathcal{F}_i)-H_k(s,\mathcal{F}_i)\Vert_2 \leq C|t-s|$ and  $\sup_{t\in[0,1]}\Vert H_k(t,\mathcal{F}_i)\Vert_4 < \infty$.
\end{lemma}

\section{Main Results}\label{sec3}
\subsection{Asymptotic theory}
By \cref{a3} and for $l=0,1,...$, define
\begin{align}
Q_{n,l}^k(t)&=\frac{1}{nb}\sum_{i=1}^n
\left(\frac{t_i-t}{b}\right)^lK
\left(\frac{t_i-t}{b}\right), \label{e9}\\
R_{n,l}^k(t)&=\frac{1}{nb}\sum_{i=1}^n \rho_i^k\left(\frac{t_i-t}{b}\right)^lK
\left(\frac{t_i-t}{b}\right). \label{e10}
\end{align} 
Then \cref{e7} can be expressed as
\begin{equation}\label{e11}
\left(\begin{array}{c}
\widehat{\beta}_{k,b}(t)\\
b\widehat{\beta}_{k,b}'(t)
\end{array}\right)=
\left(\begin{array}{cc}
Q_{n,0}^k(t) & Q_{n,1}^k(t) \\ 
Q_{n,1}^k(t) & Q_{n,2}^k(t)
\end{array}\right)^{-1}
\left(\begin{array}{c}
R_{n,0}^k(t)\\
R_{n,1}^k(t)
\end{array}\right):=[Q_n^k(t)]^{-1}R_n^k(t).
\end{equation}
Let 
$$\mu_l=\int_{\mathbb{R}}x^lK(x) \dee x~~\text{and}~~ \phi_l=\int_{\mathbb{R}}x^lK^2(x) \dee x,~l=0,1,....$$
Now, we will construct SCBs for $\beta_k(\cdot),~k=1,...,h$.

\begin{theorem}\label{thm2}
	Suppose that Assumptions \ref{a1}-\ref{a5} hold and further assume that\\
	(1) $\sigma_k(t)$ is Lipschitz continuous on [0,1].\\
	(2) $\alpha+2\beta \le 2/5$.\\
	(3) $\log(n)/(n^{3/5-\beta}b)+nb^5\log(n) \to 0$.\\
	Then, for each $k=1,...,h$, we have
	\begin{align*}
	\Pr\left[\sqrt{\frac{nb}{\phi_0}}\sup_{t\in\mathcal{T}}\left|\sigma_k^{-1}(t)\left\{\widehat{\beta}_{k,b}(t)-\beta_k(t)\right\}\right|-B_K(m^\ast)\leq \frac{u}{\sqrt{2\log(m^\ast)}}\right]\\
	=\exp\{-2\exp(-u)\},
	\end{align*}
	as $n\to\infty$, where $\mathcal{T}=[b,1-b],~m^\ast=1/b$ and 
	$$B_K(m^\ast)=\sqrt{2\log(m^\ast)}+\frac{1}{\sqrt{2\log(m^\ast)}}\log\left(\frac{1}{\pi}\sqrt{\frac{1}{4\phi_0}\int_{-1}^1|K'(u)|^2du}\right).$$
\end{theorem}

Let us comment on the conditions listed in \cref{thm2}. Condition (1) shows the smoothness of $\sigma_k(t)$. Condition (2) indicates that the change-point number and size can both go to infinity but at a slow rate. The assumption $nb^5\log(n)\rightarrow 0$ in Condition (3) is an undersmoothing requirement that reduces the bias of the estimators to the second order.

Notice that $\EE(\rho_i^k)=\beta_k(t_i)+\EE(\lambda_i^k)=2\gamma_0(t_i)-2\gamma_k(t_i)
+\tilde\Delta_i^2$, where $\tilde{\Delta}_i^2=\bigO(1/{n^2})$ when there is no change point between observations $y_i$ and $y_{i-k}$, $\tilde{\Delta}_i^2\le \Delta_n^2=\bigO(n^{2\beta})$ when there exists at least a change point on $\mu(\cdot)$. However, the estimate of  $\lambda_i^k$ can be viewed as a negligible term (see \cref{eq3} in the proof of \cref{thm1}). With the previous discussion in mind, we can define 
\begin{align*}
&\widehat{\gamma}_0(t)=\frac{1}{2}\widehat{\beta}_{h,b_h}(t),\\
&\widehat{\gamma}_k(t)=\frac{1}{2}
\left[\widehat{\beta}_{h,b_k}(t)-\widehat{\beta}_{k,b_k}(t)\right],~ 
k=1,...,h-1,
\end{align*}
where $b_h$ and $b_k$ are the bandwidths for estimators $\widehat{\gamma}_0(t)$ and $\widehat{\gamma}_k(t)$, respectively. Making it easy to distinguish, here we use the different notations for the bandwidths which will be selected by some criterion (see \cref{sec4.4}). Notice that we require the same bandwidth $(b_k)$ to compute the estimator of $\gamma_k(t)$. With the above results, the SCB for $\gamma_0(\cdot)$ is straightforward.
\begin{corollary}\label{col1}
	With the conditions in \cref{thm2}, we have
	\begin{align*}
	\Pr\left[\sqrt{\frac{4nb_h}{\phi_0}}\sup_{t\in\mathcal{T}}
	\left|\sigma_h^{-1}(t)\left\{\widehat{\gamma}_0(t)
	-\gamma_0(t)\right\}\right|
	-B_K(m^\ast)\le\frac{u}{\sqrt{2\log(m^\ast)}}\right]\\
	=\exp\{-2\exp(-u)\}.
	\end{align*}
\end{corollary}
Furthermore, to facilitate the SCB for $\gamma_k(\cdot),~k=1,...,h-1$, we will consider a linear combination of $\beta_k(\cdot)$. First, define $\widetilde{H}_k=(H_h(t,\mathcal{F}_i),
H_k(t,\mathcal{F}_i))^\top$ and a $2$ by $2$ matrix
\begin{equation}\label{sigma}
\Sigma_k^2(t)=\sum_{j=-\infty}^{\infty}{\rm Cov}(\widetilde{H}_k(t,\mathcal{F}_0)
\widetilde{H}_k(t,\mathcal{F}_j)).
\end{equation}
We also denote $\widetilde{\beta}_k(t)=(\beta_h(t),\beta_k(t))^\top$ as a two-dimensional vector, $C=(1,-1)^\top$ and $\beta_{C,k}(t)=C^\top \widetilde{\beta}_k(t).$ The natural estimators for $\widetilde{\beta}_k(t)$ and $\beta_{C,k}(t)$ are $\widehat{\beta}_k(t)=(\widehat{\beta}_{h,b_k}, \widehat{\beta}_{k,b_k})^\top$ and $\widehat{\beta}_{C,k}(t)=
C^\top\widehat{\beta}_k(t)=\widehat{\beta}_{h,b_k}-
\widehat{\beta}_{k,b_k}$, respectively. Furthermore, let $\sigma_{C,k}^2(t)=C^\top\Sigma^2_k(t)C,$ similar to  Theorem 3 in \cite{ZW2010}. At this point, we can obtain the following result.

\begin{corollary}\label{col2}
    Suppose that the smallest eigenvalue of $\sigma_{C,k}(t)$ is bounded away from 0 on $[0,1]$ for $k=1,...,h-1.$ Moreover, we assume that all of the conditions of \cref{thm2} are valid. Then, we have (i)
	\begin{align*}
		\Pr\left[\sqrt{\frac{nb_k}{\phi_0}}\sup_{t\in\mathcal{T}}
		\left|\sigma_{C,k}^{-1}(t)\left\{\widehat{\beta}_{C,k}(t)
		-\beta_{C,k}(t)\right\}\right|
		-B_K(m^\ast)\leq \frac{u}{\sqrt{2\log(m^\ast)}}\right]\\
	=\exp\{-2\exp(-u)\},
	\end{align*}
	as $n\to\infty$. (ii) Furthermore, one can easily deduce the SCB for $\gamma_k(\cdot)$, $k=1,...,h-1$,
	\begin{align*}
	\Pr\left[\sqrt{\frac{4nb_k}{\phi_0}}\sup_{t\in\mathcal{T}}\left|\sigma_{C,k}^{-1}(t)\left\{\widehat{\gamma}_k(t)-\gamma_k(t)\right\}\right|-B_K(m^\ast)\leq \frac{u}{\sqrt{2\log(m^\ast)}}\right]\\
	=\exp\{-2\exp(-u)\}.
	\end{align*}
\end{corollary}

\begin{remark}
	It is noteworthy to mention that for estimating $\beta_{C,k}(t)$, we use the same bandwidth $b_k$; therefore, the entire estimator $\widehat{\beta}_{C,k}(t)=\widehat{\beta}_{h,b_k}(t)-\widehat{\beta}_{k,b_k}(t)$ depends on only a single tuning parameter (bandwidth $b_k$). This enables us to achieve the conclusion of \cref{col2}(i) based on the result of \cref{thm2}. As a result, \cref{col2}(ii) also holds true due to this fact.  	
\end{remark}

After constructing SCBs for the second-order structure $\gamma_k(\cdot)$, the following theorem states that $\widehat{\gamma}_k(t)$ are consistent estimators for $\gamma_k(t)$ uniformly in $t$ for all $k=0,...,h-1$.

\begin{theorem}\label{thm1}
	Under Assumptions \ref{a1}-\ref{a5} and suppose conditions 
	$$\alpha+2\beta\le \frac{2}{5},~~~\frac{\log(n)}{n^{3/5-\beta}b}+
	nb^5\log(n)\to 0$$ hold true. Then, we have
	\begin{equation*}
	\sup_{t\in \mathcal{T}} \left|\widehat{\beta}_{k,b}(t)
	-\beta_k(t)\right|=\bigO_{\Pr}(\chi_n),~~k=1,...,h,
	\end{equation*}
	where $\chi_n=b^2+
	\frac{\log(n)}{n^{3/5-\beta}b}+
	\sqrt{\frac{\log(n)}{n^{1-2\beta}b}}+
	\sqrt{\frac{1}{n^{1-\alpha-4\beta}b}}$.
\end{theorem}

This theorem implies the uniform consistency of $\widehat{\beta}_{k,b}(\cdot).$ Additionally, due to the relationship between $\beta_k$ and $\gamma_k,$ we can also easily obtain the following consistency result for $\widehat\gamma_k(\cdot).$

\begin{corollary}
	With the conditions in \cref{thm1}, we have 
	\begin{equation*}
	\sup_{t\in\mathcal{T}}\left|\widehat{\gamma}_k(t)-\gamma_k(t)
	\right|=\bigO_{\Pr}(\chi_n),~~k=0,...,h-1.
	\end{equation*}
\end{corollary}

\section{Practical implementation}\label{sec4}
\subsection{Selection of the difference lag}\label{sec4.1}
Note that for any fixed time $t$, $\widehat{\beta}_{k,b_k}(t)= 2\widehat{\gamma}_0(t)-2\widehat{\gamma}_k(t)$ and recall that $\gamma_k(t)\approx 0$ when $k\ge h$, where $h$ is a large value that has been chosen in advance. Hence, we know that if $k\ge h$, $\widehat{\beta}_{k,b_k}(t)\approx 2\widehat{\gamma}_0(t)$ is practically invariant with respect to $t$ as $k$ increases. This fact suggests the following bandwidth selection procedure.

First, for any fixed $t$, we choose a large enough value $h_0(t)$ and select $k=h_0(t).$ Next, we  calculate $\widehat{\beta}_{k,b_k}(t)$. Then, by successively decreasing the value of $k$ and considering  $k=h_0(t)-1,h_0(t)-2,...,$ we calculate the corresponding quantities $\widehat{\beta}_{k,b_k}(t)$ until $\widehat{\beta}_{k,b_k}(t)$ shows an abrupt change. At this point, the optimal difference lag for time $t$ can be selected as the current $k$ plus $1.$ Intuitively, we can interpret this through the scatterplot of $(k,\widehat{\beta}_{k,b_k}(t))$. When the slope of the function $\widehat{\beta}_{k,b_k}(t)$ shows an obvious change, then we can choose $h^\ast(t)=k+1$. Following the above procedure for each time point $t_i=i/n,~i=1,...,n$, we finally choose the optimal lag as $h=\sum_{i=1}^n h^\ast(t_i)/n$. 

\subsection{Covariance matrix estimation}\label{sec4.2}
To apply Corollaries \ref{col1} and \ref{col2} (ii), we need to estimate the long-run variance $\Sigma_k^2(\cdot)$ in \cref{sigma} first. This problem is complicated but has been extensively studied by many researchers. Here we adopt the technique considered by \cite{ZW2010}.

Let $Q_i^k=\sum_{j=-m}^m \tilde{\eps}_{i+j},~i=1,...,n$, where $\tilde{\eps}_i=(\eps_i^h,\eps_i^k)^\top$ for $i=1,...,n$. Notice that $\EE(\tilde{\eps}_i)=0$ and denote $N_i^k:=Q_i^kQ_i^{k\top}/(2m+1)$. In the locally stationary case, we can make use of the fact that a block of $\{\tilde{\eps}_i\}$ is approximately stationary when its length is small compared with $n$. Hence, $\EE(N_i^k)\approx \Sigma_k^2(t_i)$ as $m\to \infty$ and $m/n\to 0$. Let $\tau=\tau_n$ be the bandwidth and define the covariance matrix estimator as
$$\widetilde{\Sigma}_k^2(t)=\sum_{i=1}^n\tilde{\omega}_n^\tau(t,i)N_i^k,~~
\tilde{\omega}_n^\tau(t,i)=\frac{K_\tau(t_i-t)}{\sum_{k=1}^n K_\tau(t_k-t)},$$
	with $\tau$ being the bandwidth. Therefore, the estimate $\widetilde{\Sigma}_k^2(t)$ is guaranteed to be positive semidefinite. The following theorems provide consistency of our covariance matrix estimate.

\begin{theorem}\label{thm4}
	Assume that $\Sigma_k^2(t)\in \mathcal{C}^2[0,1],~\delta_4(\widetilde{H}_k,j)=\bigO(\{j\log(j)\}^{-2}),~m=m_n\to \infty,~m=\bigO(n^{1/3}),~\tau\to 0$ and $n\tau\to \infty$. Then, $(i)$ for each $k,~k=1,...,h-1$ and any fixed $t\in (0,1)$, 
	$$\left\Vert\widetilde{\Sigma}_k^2(t)-\Sigma_k^2(t)\right\Vert=\bigO\left(\sqrt{ \frac{m}{n\tau}}+\frac{1}{m}+\tau^2\right),$$
	$(ii)$ for $\mathcal{I}=[\tau,1-\tau]$, 
	$$\left\Vert\sup_{t\in\mathcal{I}}\left|\widetilde{\Sigma}_k^2(t)
	-\Sigma_k^2(t)\right|\right\Vert=\bigO\left(\sqrt{
		\frac{m}{n\tau^2}}+\frac{1}{m}+\tau^2\right).$$
\end{theorem}

In practice, the errors $\tilde{\eps}_j$ cannot be observed, thus we use $\widehat{\Sigma}_k^2(t)=\sum_{i=1}^n\widetilde\omega_n^\tau(t,i)\widehat{N}_i^k$, where $\widehat{N}_i^k$ is defined as $N_i^k$ with $\tilde{\eps}_j$ therein replaced by its estimator $\hat{\eps}_j$.

\begin{theorem}\label{thm5}
	Assume that conditions of \cref{thm1} and conditions of \cref{thm4} hold. Denote $\nu_n=\sqrt{m\log(n)}\chi_n$, where $\chi_n$ is defined in \cref{thm1} and further assume $\nu_n\to 0$. Then 
	$$\sup_{t\in\mathcal{I}}\left|\widehat{\Sigma}_k^2(t)
	-\Sigma_k^2(t)\right|=\bigO_{\Pr}\left(\nu_n+\sqrt{
		\frac{m}{n\tau^2}}+\frac{1}{m}+\tau^2\right).$$
\end{theorem}

Note that $\sigma_h^2(t)$ is the first diagonal element of $\Sigma_k^2(t)$ and $\sigma_{C,k}^2(t)=C^\top\Sigma_k^2(t)C$. Thus, the covariance estimates in Corollaries \ref{col1} and \ref{col2} (ii) can be easily calculated via plugging in the long-run covariance matrix estimate $\widehat{\Sigma}_k^2(t)$.

\subsection{Simulation assisted bootstrapping method}
Now we aim to apply \cref{col1} and \cref{col2} (ii) to construct the SCBs. Let $\widehat{\gamma}_0''(t)$ and $\widehat{\gamma}_k''(t)$ be uniformly consistent estimators of $\gamma_0''(t)$ and $\gamma_k''(t)$ for $k=1,...,h-1$, respectively. Then the corresponding $(1-\alpha)$th SCB with $\alpha\in(0,1)$ for $\gamma_0(t)$ and $\gamma_k(t)$ are
\begin{align*}
&\left[\widehat{\gamma}_0(t)\pm \hat{\sigma}_h(t)
\sqrt{\frac{\phi_0}{4nb}}\left(B_K(m^\ast)-
\frac{\log[\log(1-\alpha)^{-1/2}]}
{\sqrt{2\log(m^\ast)}}\right)\right],\\
&\left[\widehat{\gamma}_k(t)\pm \hat{\sigma}_{C,k}(t)\sqrt{\frac{\phi_0}{4nb}}
\left(B_K(m^\ast)-\frac{\log[\log(1-\alpha)^{-1/2}]}
{\sqrt{2\log(m^\ast)}}\right)\right],~k=1,...,h-1.
\end{align*}

Due to the slow rate of convergence to Gumbel distribution, in practice, the UCB from Corollaries \ref{col1} and \ref{col2} (ii) may not have good finite-sample performances. To circumvent this problem, we shall adopt a simulation assisted bootstrapping approach. 

\begin{proposition}\label{prop1}
	Suppose conditions in \cref{thm2} hold and also assume that $\sigma_k(t)$ is Lipschitz continuous for $k=1,...,h$. Then, on a richer probability space, there are i.i.d. standard normal distributed random variables $u_i$ such that
	\begin{align*}
	&\sup_{t\in\mathcal{T}}|\widehat{\gamma}_0(t)-
	\gamma_0(t)-Z_0(t)|=\bigO_{\Pr}(\psi_n),\\
	&\sup_{t\in\mathcal{T}}|\widehat{\gamma}_k(t)-\gamma_k(t)
	-Z_{C,k}(t)|=\bigO_{\Pr}(\psi_n),~k=1,...,h-1,
	\end{align*}
	where $\psi_n=n^{1/2}b^{7/2}+\sqrt{\frac{1}{n^{1-2\alpha-4\beta}b}}
	+\sqrt{\frac{\log(n)}{n^{1-2\alpha-2\beta}b}}+
	\sqrt{\frac{b\log(n)}{n}},~Z_0(t)=\sigma_h(t)\mu_{b_h}^\dagger(t)$ and 
	$Z_{C,k}(t)=\sigma_{C,k}(t)\mu_{b_k}^\dagger(t)$ with $\mu_b^\dagger(t)=\sum_{i=1}^n\omega_n^b(t,i)u_i/2$.
\end{proposition}  
The proposition implies that the distribution of $\sup_{t\in\mathcal{T}}|\sigma_h^{-1}(t)[\widehat{\gamma}_0(t)-\gamma_0(t)]|$ ($\sup_{t\in\mathcal{T}}|\sigma_{C,k}^{-1}(t)$\\ $[\widehat{\gamma}_k(t)-\gamma_k(t)]|$) can be well approximated by that of $\sup_{t\in \mathcal{T}}|\mu_{b_h}^\dagger(t)|~ (\sup_{t\in \mathcal{T}}|\mu_{b_k}^\dagger(t)|)$, which can be obtained by generating a large number of i.i.d. copies via bootstrapping. Therefore, the above proposition provides us with an alternative way to construct the SCB of the autocovariance function without using the asymptotic Gumbel distribution. 

For ease of application, we combine procedures mentioned above into a convenient sequence of steps below. 
\begin{itemize}
	\item Choose the difference lag order $h$ by using method that is proposed in \cref{sec4.1}.
	\item  Find appropriate bandwidths $b_h$ and $b_k$ for estimating $\beta_h(\cdot),~\beta_h(\cdot)-\beta_k(\cdot)$ respectively, and the bandwidth $\tau$ for estimating $\Sigma_k^2(\cdot)$.
	\item Generate i.i.d. random variables $u_1,u_2,...\sim N(0,1)$ and calculate $\sup_{t\in[b_k,1-b_k]}|\mu_{b_k}^\dagger(t)|$ for $k=1,...,h$.
	\item Repeat the last step for a large number of times (e.g. $10^4$ ) and obtain the estimated $(1-\alpha)$th quantile $\hat{q}_{1-\alpha}^k$ of $\sup_{t\in [b_k,1-b_k]}|\mu_{b_k}^\dagger(t)|$.
	\item Calculate $\widehat{\Sigma}_k^2(t)$ by using the method in \cref{sec4.2}. Then, obtaining $\widehat{\sigma}_h(t)$ together with $\widehat{\sigma}_{C,k}(t)$ is straightforward. 
	\item Construct the $(1-\alpha)$th SCB of the auto-covariance function $\gamma_0(t)$ as $\widehat{\gamma}_0(t)\pm \hat{q}_{1-\alpha}^h\hat{\sigma}_h(t)$, and $\widehat{\gamma}_k(t)\pm \hat{q}_{1-\alpha}^k\widehat{\sigma}_{C,k}(t)$ for $\gamma_k(t),~k=1,...,h-1$.
\end{itemize}

\subsection{Selection of tuning parameters}\label{sec4.4}
In this subsection, we briefly discuss the practical choices of tuning parameters $b,~m$ and $\tau$. Here, we consider the generalized cross-validation (GCV) method by \cite{Craven1978} to choose the bandwidth $b$. Specifically, we consider two cases of bandwidth selection for $\gamma_0(t)$ and $\gamma_k(t),~k=1,...,h-1$, respectively. For estimating  $\gamma_0(t)$, let $P_h=(\rho_1^h,...,\rho_n^h)^\top$ and $\widehat{P}_h(b)=(\hat{\rho}_1^h(b),...,\hat{\rho}_n^h(b))^\top$ be the corresponding fitted values. One can write $\widehat{P}_h(b)=H(b)P_h$, where $H(b)$ is an $n$ by $n$ square hat matrix that depends on $b$. Then, we choose the optimal bandwidth (say $b_h$) that minimizes 
$${\rm GCV}_h(b)=\frac{n^{-1}\sum_{i=1}^n[\rho_i^h-\hat{\rho}_i^h(b)]^2}
{[1-{\rm trace}(H(b))/n]^2}.$$

On the other hand, when estimating $\gamma_k(t)$ for $k=1,...,h-1$, we treat $\widehat{\gamma}_k(t)$ as a whole term and choose a joint bandwidth for it. Similarly, denote $P_k=(\rho_1^h-\rho_1^k,...,\rho_n^h-\rho_n^k)^\top$ and let $\widehat{P}_k(b)=(\hat{\rho}_1^h(b)-\hat{\rho}_1^k(b),...,\hat{\rho}_n^h(b)-
\hat{\rho}_n^k(b))^\top$ be the corresponding fitted values. As before, one can write $\widehat{P}_k(b)=H(b)P_k.$ With this in mind, we select as optimal the bandwidth (say $b_k$) that minimizes the following quantity: 
$${\rm GCV}_k(b)=\frac{n^{-1}\sum_{i=1}^n[(\rho_i^h-\rho_i^k)-(\hat{\rho}_i^h(b)
	-\hat{\rho}_i^k(b))]^2}{[1-{\rm trace}(H(b))/n]^2}.$$

For the choice of $m$ and $\tau$, we now employ the extended minimum volatility method (including two parameters) which was proposed in \cite[Chapter 9]{Politis1999}. This method is based on the fact that if a pair of block size and bandwidth is in an appropriate range, then confidence regions for the local mean constructed by $\widehat{\Sigma}_k^2(t)$ should be stable. Therefore, we first consider a grid of possible block sizes and bandwidths and then choose the optimal pair that minimizes the volatility of the boundary points of the confidence regions in the neighborhood of this pair. To be more specific, let the grid of possible block sizes and bandwidths be $\{m_1,...,m_{M_1}\}$ and $\{\tau_1,...,\tau_{M_2}\}$, respectively. Then denote the estimated long-run covariance matrices as $\{\widehat{\Sigma}_k^2(m_i,\tau_j,t)\}$ for $i=1,...,M_1,~j=1,...,M_2$. For each pair $(m_i,\tau_j)$, we need to calculate 
\begin{equation}\label{ise}
{\rm ISE}\left[\cup_{r=-2}^2\{\widehat{\Sigma}_k^2(m_{i+r},\tau_j,t)\} \cup \cup_{r=-2}^2\{\widehat{\Sigma}_k^2(\tau_{j+r},m_i,t)\}\right],~k=1,...,h,
\end{equation}
where ISE denotes the integrated standard error
$${\rm ISE}[\{\widehat{\Sigma}_k^2(s,\cdot,t)\}_{s=1}^l]=\int_0^1\left\{
\frac{1}{l-1}\sum_{s=1}^l\left\vert\widehat{\Sigma}_k^2(s,\cdot,t)-
\bar{\widehat{\Sigma}}_k^2(\cdot,t)\right\vert^2\right\}^{1/2}\dee t$$
with $\bar{\widehat{\Sigma}}_k^2(\cdot,t)=\sum_{s=1}^l\widehat{\Sigma}_k^2
(s,\cdot,t)/l$ and $s$ being the parameter $m$ or $\tau$. Finally, we choose the pair $(m_i^\ast,\tau_j^\ast)$ that minimizes \cref{ise}.

\section{Simulations}\label{sec_sim}

To illustrate performance of the proposed estimator of autocovariance, we consider several models.  For each model, we obtain the uniform confidence interval coverage of the true variance function and the autocovariance function at lag $1$ for three different sample sizes: $n=400,$ $n=600$ and $n=800.$ In each case, we use $500$ replications. To select bandwidths $b_h$ and $b_k$ we use the grid from $0.15$ to $0.45$ with the step size $0.01.$ We also provide a graphical illustration of a confidence interval enclosing the true variance and autocovariance lag $1$ functions for each of the models considered. 

The first model considered has the errors that are generated by a locally stationary linear process \eqref{loc.stat.linear} with $a_{j}(t)=\left(\frac{t}{2}\right)^{j},$ $j=1,2,\ldots$ while the sequence $(\zeta_{i})$ consists of iid normal random variables with mean zero and variance $1.$ In this case the coefficients start with $j=1$ since otherwise $a_{0}(t)$ is undefined at $t=0.$ The Assumption $1$ is satisfied since $\sup_{t\in [0,1]}a_{j}^{2}(t)\le \left(\frac{1}{4}\right)^{j}$ which is, of course, $O(j^{-2}).$ Assumption $2$ is also satisfied because $\sum_{j=1}^{\infty} \sup_{t\in [0,1]} a_{j}^{2}(t)=\sum_{j=1}^{\infty} \frac{j^{2}}{4^{j}}<\infty.$ Next, the mean function $\mu(t)$ is taken to be a piecewise constant function with six change-points located at fractions $\frac{1}{6}\pm\frac{1}{36},$ $\frac{3}{36}\pm \frac{2}{36},$ and $\frac{5}{6}\pm \frac{3}{36}$ of the sample size $n$. In the first segment, $\mu_{1}\equiv \mu(t_{1})=0,$ in the second it is equal to $1,$ and in the remaining segments $\mu(t)$ alternates between $0$ and $1,$ starting with $0$ in the third segment. This mean function is very similar to the one that has been considered earlier in several other publications; see e.g. \cite{Chakar.etal.16} and \cite{tecuapetla2017autocovariance}. 

The second model we consider has exactly the same error structure as Model $1$ but the mean function is a slightly different one.  In particular, we make the value of the function in the second segment $2$ instead of $1$ while the remaining segments of $\mu(t)$ alternate between $0$ and $1,$ starting with $0$ in the third segment. Since the error process remains the same as before in Model $1,$ Assumptions $1$ and $2$ are satisfied.

The third model we consider is where the errors are generated by a locally stationary MA($2$) process
\[x_{i}=\sum_{j=0}^{2}a_{j}(t)\zeta_{i-j}\] with coefficients being equal to $a_{j}(t)=\frac{(t+0.05)^{j}}{2^{j}},$ $j=0,1,2.$ The sequence $(\zeta_{i})$ consists of iid $N(0,0.3)$ random variables. The locally stationary MA process considered is a special case of the general locally stationary linear process. Since the process consists of the finite number of terms, the stochastic Lipschitz continuity condition in the Assumption $2$ is satisfied automatically. Because $\sup_{t\in [0,1]}a_{j}^{2}(t)\le (0.525)^{2j},$ the Assumption $1$ will also be satisfied. Finally, the mean function stays the same as in the Model $1.$ 

In \cref{table400,table600,table800}, we illustrate coverage probabilities of uniform confidence intervals of the variance function and lag $1$ covariance function for all three of the models considered.  We also consider three possible sample sizes, $n=400,$ $n=600$ and $n=800.$ Note that even a relatively small sample size of $400$ gives excellent coverage probabilities. It is also worthwhile noting that the coverage probabilities are generally higher for lag $1$ autocovariance function than for the variance function. 

\begin{table}[htbp!]
\centering
\begin{tabular}{|c|c|c|c|}
\hline
& Model 1 &Model 2 & Model 3\\
\hline
Variance &0.968 &0.994 & 0.962\\
\hline
Lag $1$ autocovariance &0.994 &0.998 &0.998\\
\hline
\end{tabular}
\caption{Empirical coverage probabilities for all the models when the sample size is 400}
\label{table400}
\end{table}

\begin{table}[htbp!]
\centering
\begin{tabular}{|c|c|c|c|}
\hline
& Model 1 &Model 2 & Model 3\\
\hline
Variance &0.980 &0.996 & 0.962\\
\hline
Lag $1$ autocovariance &0.996 & 1.000&0.998\\
\hline
\end{tabular}
\caption{Empirical coverage probabilities for all the models when the sample size is 600}
\label{table600}
\end{table}

\begin{table}[htbp!]
\centering
\begin{tabular}{|c|c|c|c|}
\hline
& Model 1 &Model 2 & Model 3\\
\hline
Variance &0.996 & 0.970 &0.972 \\
\hline
Lag $1$ autocovariance &0.996 & 0.996 &0.998\\
\hline
\end{tabular}
\caption{Empirical coverage probabilities for all the models when the sample size is 800}
\label{table800}
\end{table}

To illustrate the behavior of uniform confidence intervals for each of the three models considered, we also include sample plots of fitted variance/autocovariance curves with corresponding confidence intervals. For each model, two plots are given: one with the true variance function, its estimate, and a uniform confidence interval for the estimated variance curve, while the other one contains the true lag $1$ autocovariance function, its estimate, and the corresponding uniform confidence interval for the estimated autocovariance curve.  In each of the plots, a solid line is used for the true variance/autocovariance curve, a dashed line for the corresponding estimated curve, and red dotted lines for uniform confidence intervals. 

\begin{figure}[htbp!] 
	\begin{minipage}[t]{0.5\linewidth} 
		\centering 
		\includegraphics[width=7cm,height=5cm]{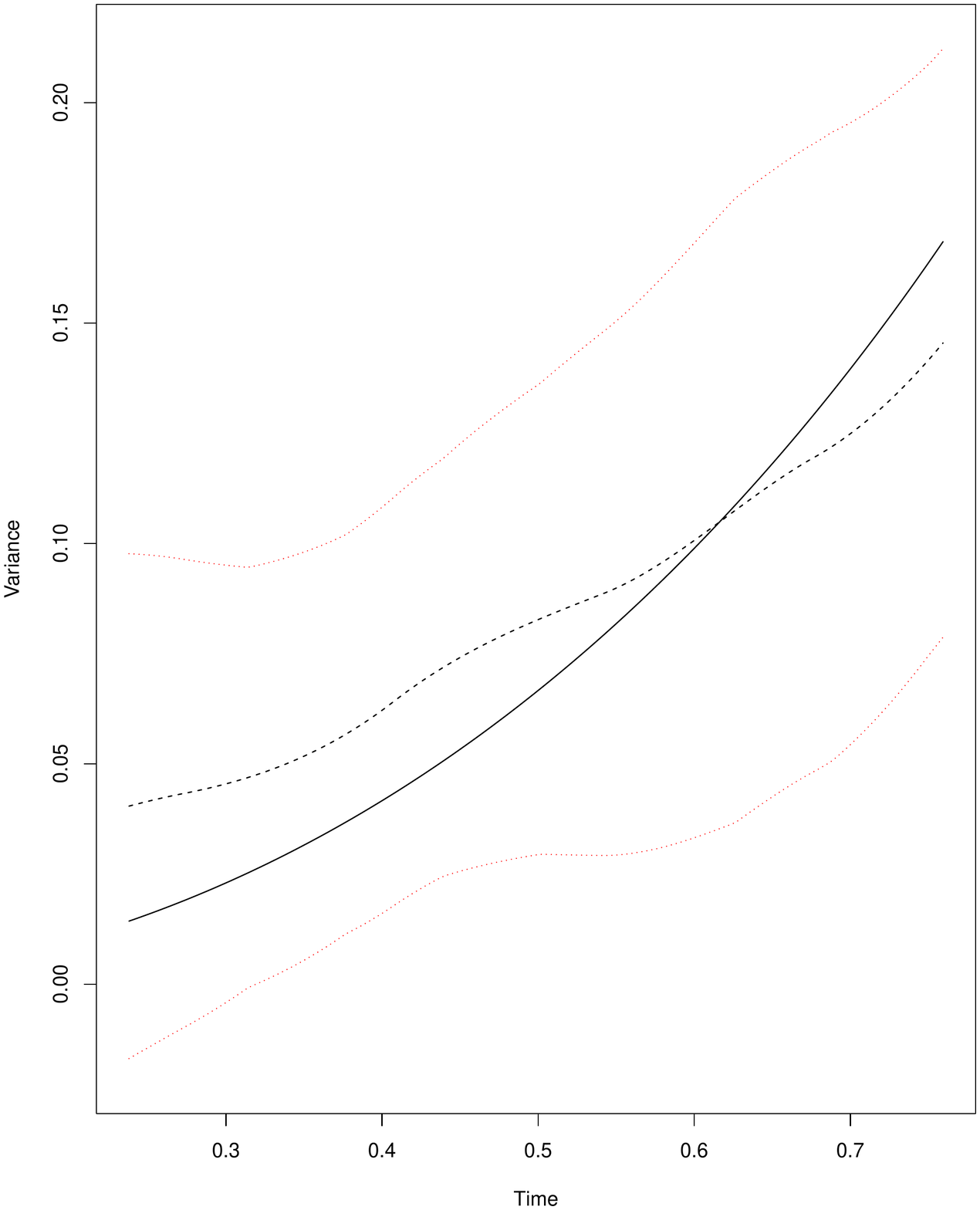}
	\end{minipage}
\label{fig:sub1} 
	\begin{minipage}[t]{0.5\linewidth} 
		\centering 
		\includegraphics[width=7cm,height=5cm]{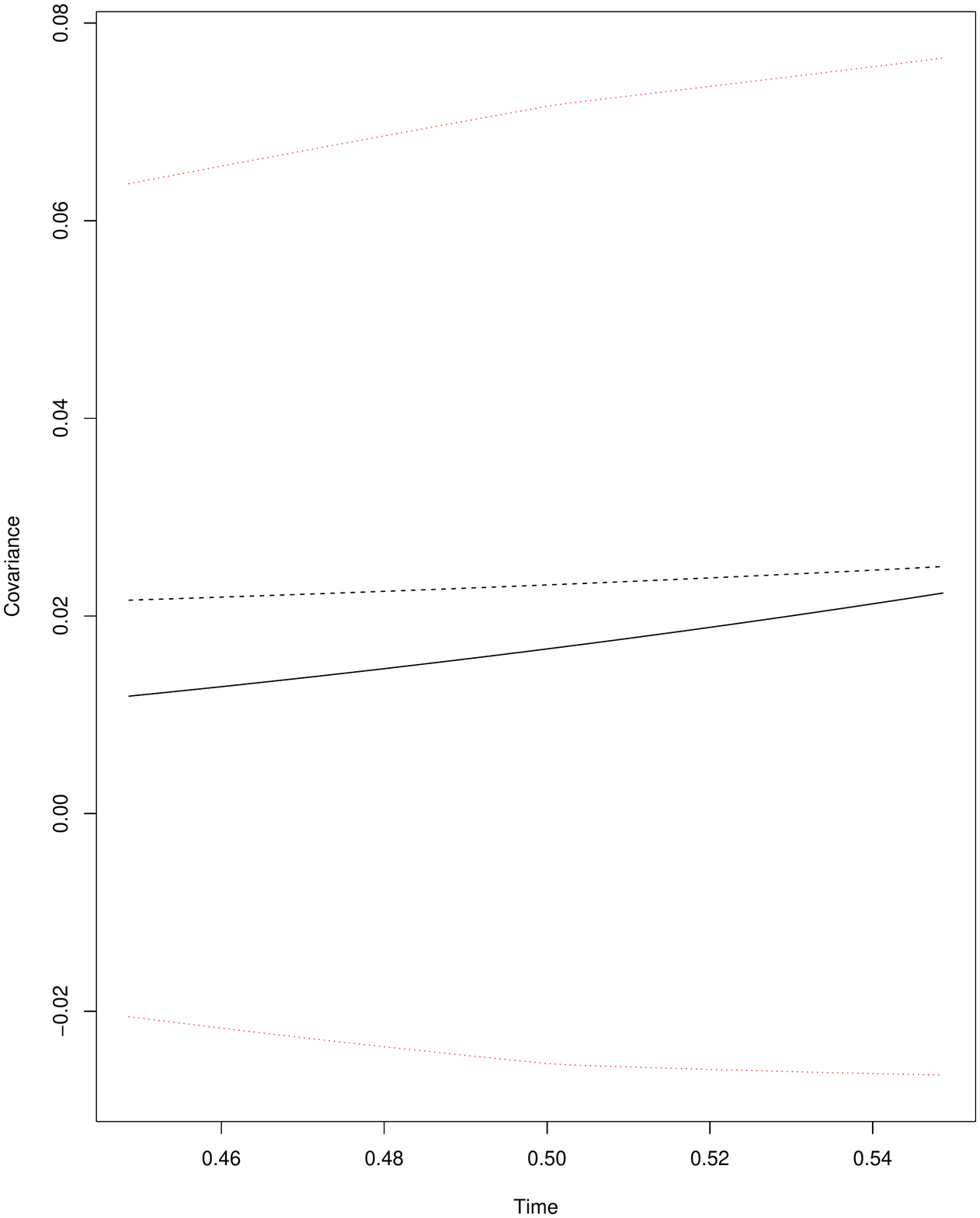} 
	\end{minipage}
\label{fig:sub2} 
\caption{Estimated variance (left) and lag $1$ autocovariance functions (right) for Model $1$}
\label{fig:test}
\end{figure}

\begin{figure}[htbp!] 
	\begin{minipage}[t]{0.5\linewidth} 
		\centering 
		\includegraphics[width=7cm,height=5cm]{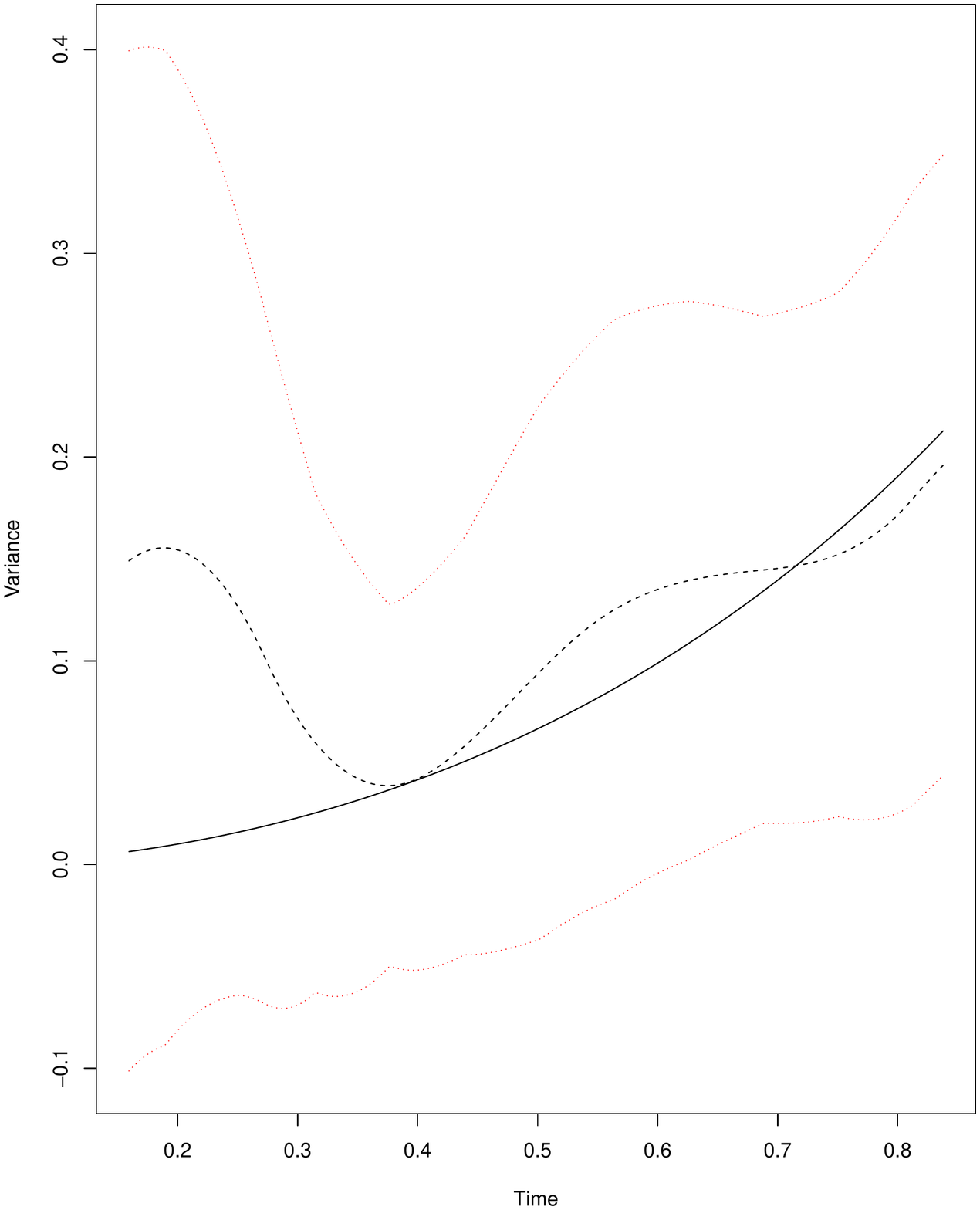}
	\end{minipage}
	\label{fig:sub11} 
	\begin{minipage}[t]{0.5\linewidth} 
		\centering 
		\includegraphics[width=7cm,height=5cm]{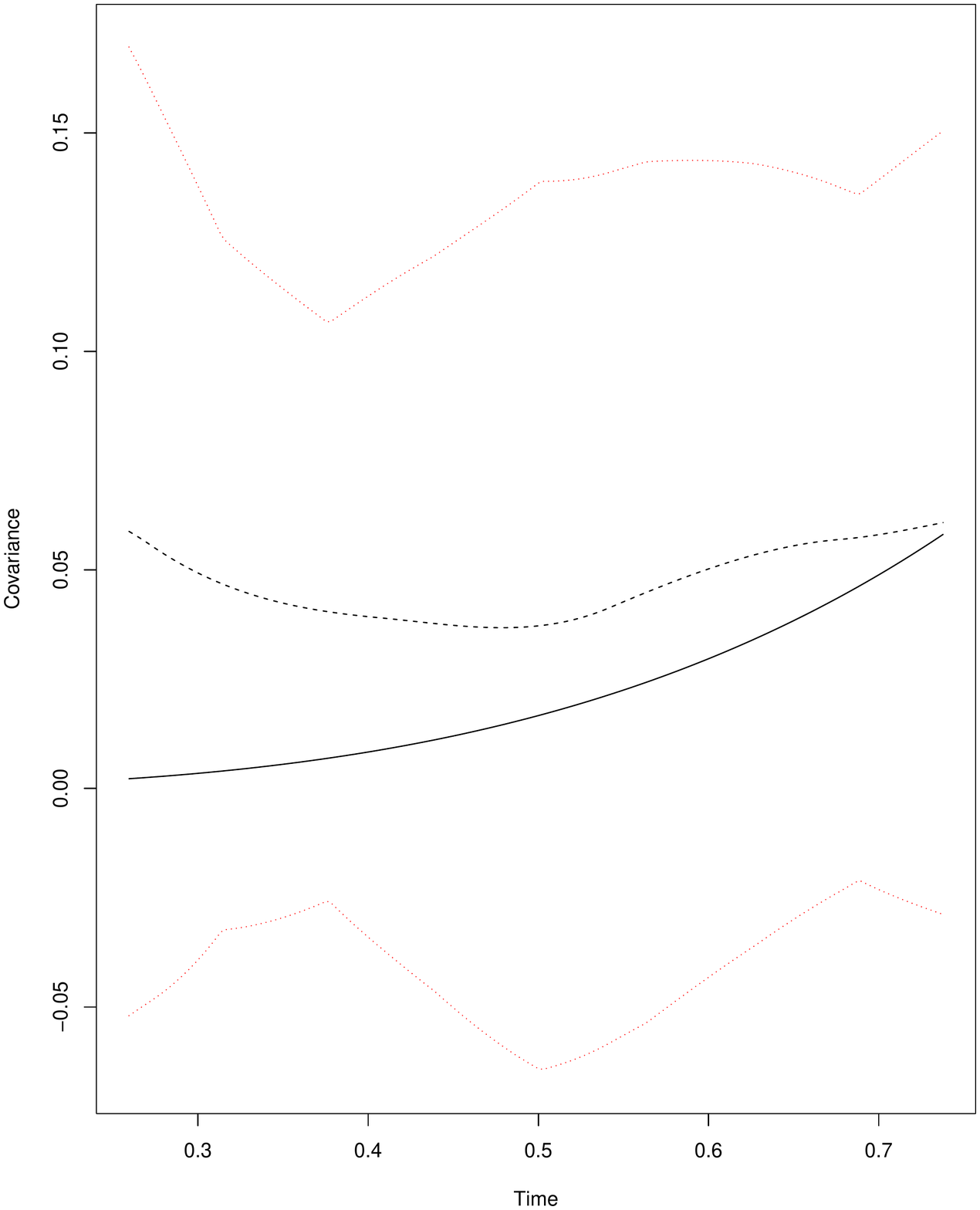} 
	\end{minipage}
	\label{fig:sub12} 
	\caption{Estimated variance (left) and lag $1$ autocovariance functions (right) for Model $2$}
	\label{fig:test1}
\end{figure}

\begin{figure}[htbp!] 
	\begin{minipage}[t]{0.5\linewidth} 
		\centering 
		\includegraphics[width=7cm,height=5cm]{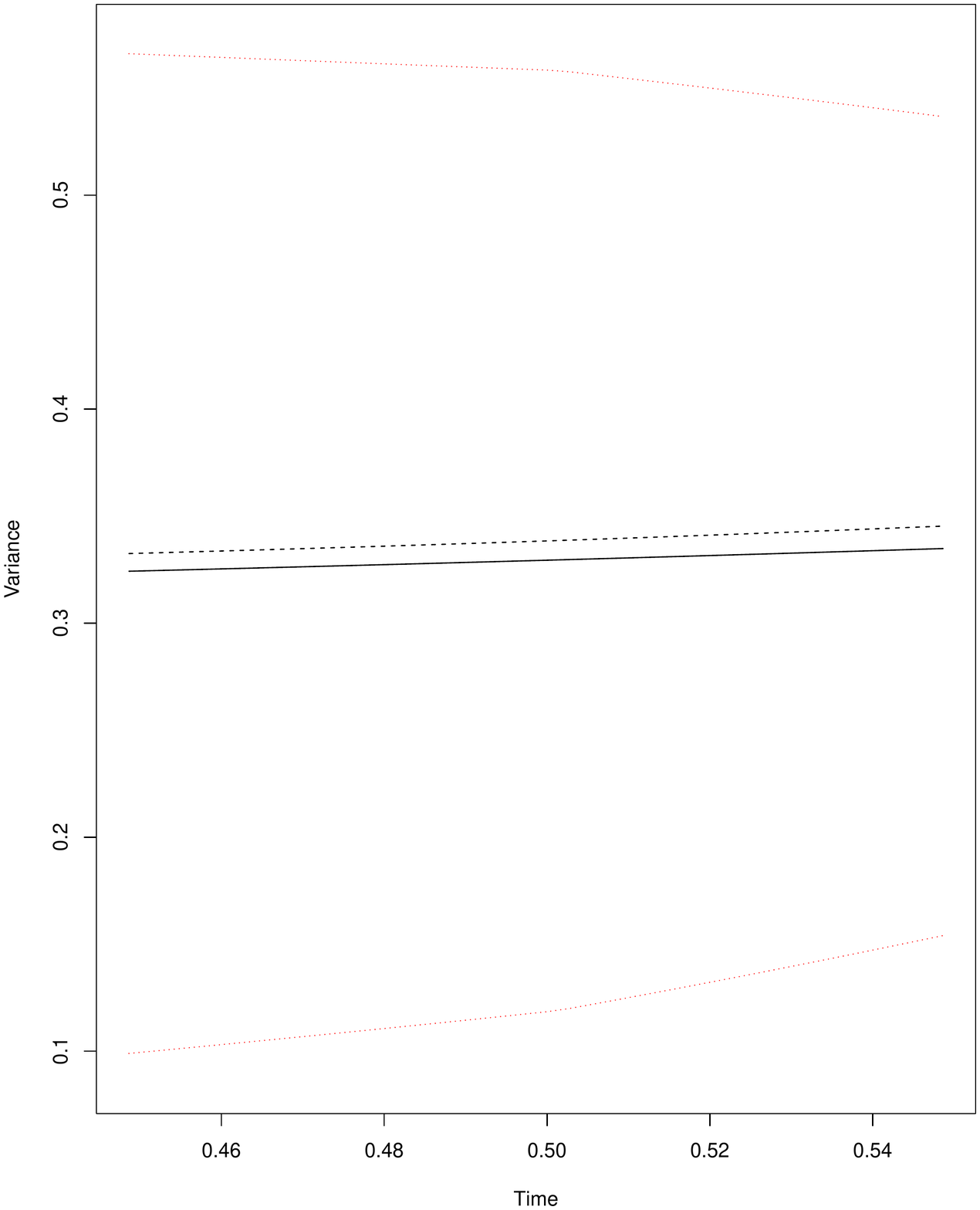}
	\end{minipage}
	\label{fig:sub13} 
	\begin{minipage}[t]{0.5\linewidth} 
		\centering 
		\includegraphics[width=7cm,height=5cm]{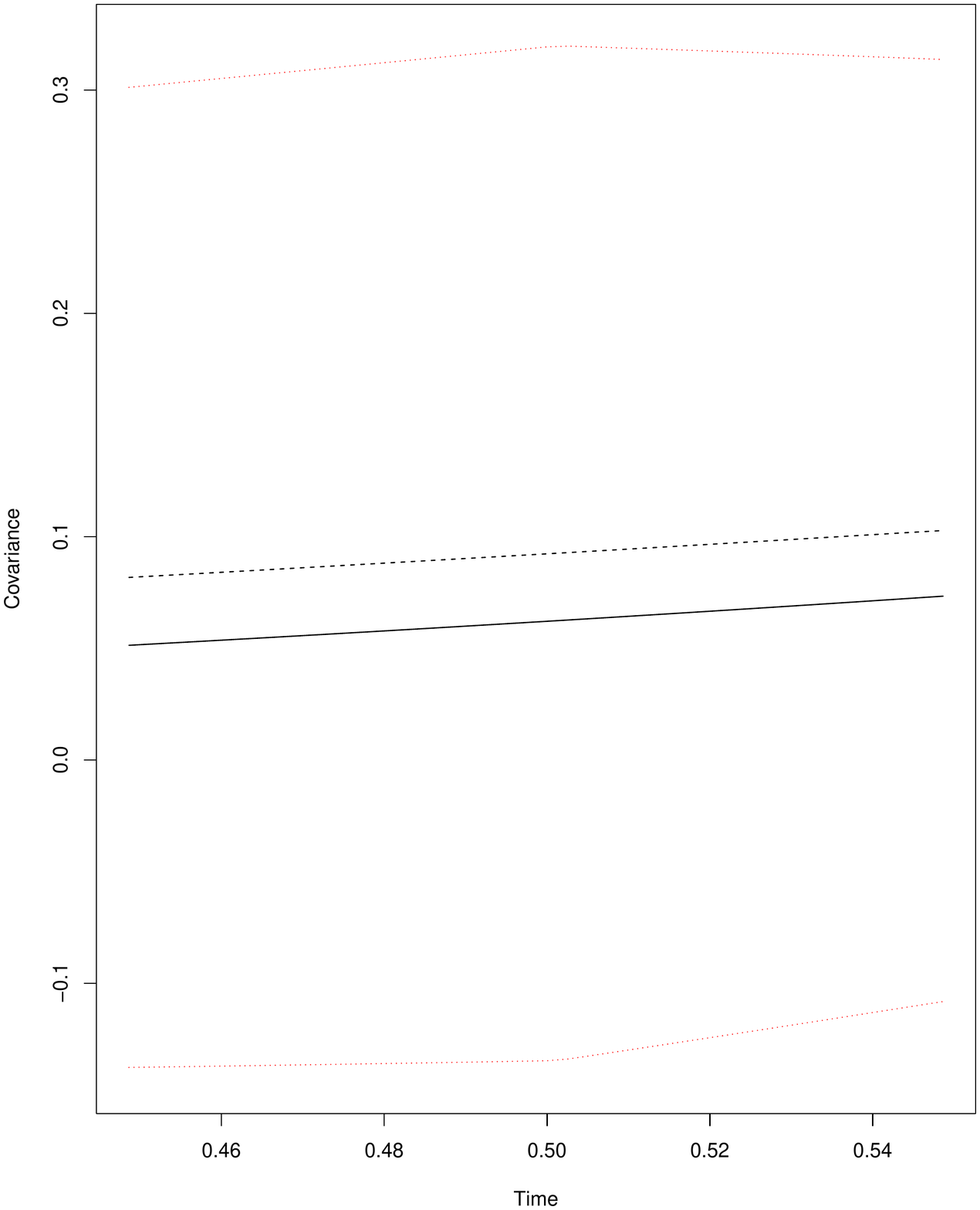} 
	\end{minipage}
	\label{fig:sub14} 
	\caption{Estimated variance (left) and lag $1$ autocovariance functions (right) for Model $3$}
	\label{fig:test2}
\end{figure}

To illustrate the reasonableness of our method, we also provide a quick comparison of our approach to a very straightforward ``naive" method. Such a method would start with a rough estimate of the mean function $\mu(t)$ using a local smoother, for example, a local linear regression. The resulting rough estimate of the mean function can then be subtracted from observations to form a series of residuals $e_{i},$ $i=1,\ldots,n.$  Using this series, a naive approach would estimate the variance function $\gamma_{0}(t)$ by applying a smoother, e.g. yet again the local linear regression, to squared residuals $e^{2}_{i},$ $i=1,\ldots,n.$ In much the same way, applying the local linear regression to a series $e_{i}e_{i-1},$ $i=2,\ldots,n$ will result in a ``naive" estimate of the lag $1$ autocovariance function $\gamma_{1}(t).$ In both situations, we used a simple generalized cross-validation to obtain the optimal smoothing bandwidth. 

It is probably sufficient to say that such a naive approach fails completely in an attempt to estimate the second order structure when the mean is discontinuous and has numerous change points. More specifically, we tried to obtain the coverage of the true variance function $\gamma_{0}(t)$ by a uniform confidence interval that is based on the ``naive" estimate described above. To do so, we used our Model $1$ with the sample size $n=400$ and $500$ replications. We found that the coverage is zero, that is, the true variance function $\gamma_{0}(t)$ is never completely inside the uniform confidence interval. This can be explained properly by noticing that our mean estimate used to obtain residuals is extremely crude. More specifically, in order for the local linear mean function estimator to be consistent at a given point the mean function $\mu(t)$ has to have two continuous derivatives at that point; see e.g. \cite{fan1994gijbels} p. $62$ for a detailed discussion. This lack of consistency results in a severe bias of the variance function estimator. Thus, such a direct approach seems to be completely inappropriate for determination of the second order structure.  

\section{Real data application}\label{sec6}

In this section, we illustrate our approach using a real dataset. There is a rather clear evidence that the global temperatures are nonstationary (see e.g. \cite{rao2004nonstationary}) and so we use the dataset that consists of monthly temperature anomalies observed during the period from January $1856$ to September $2019.$ A shorter subset of the same series has been used earlier in \cite{rao2004multiple}. The data used are publicly available from the Climate Research Unit of the University of East Anglia, UK at https://www.cru.uea.ac.uk/. The anomalies are defined here as the difference of temperatures from a reference value. The anomaly data are available for both Northern and Southern hemisphere separately. Figures \eqref{fig:nhemdata} and \eqref{fig:shemdata} display the temperature anomaly data for both hemispheres. 
\begin{figure}[h]
\centering
\begin{subfigure}{.5\textwidth}
  \centering
  \includegraphics[width=7cm,height=7cm]{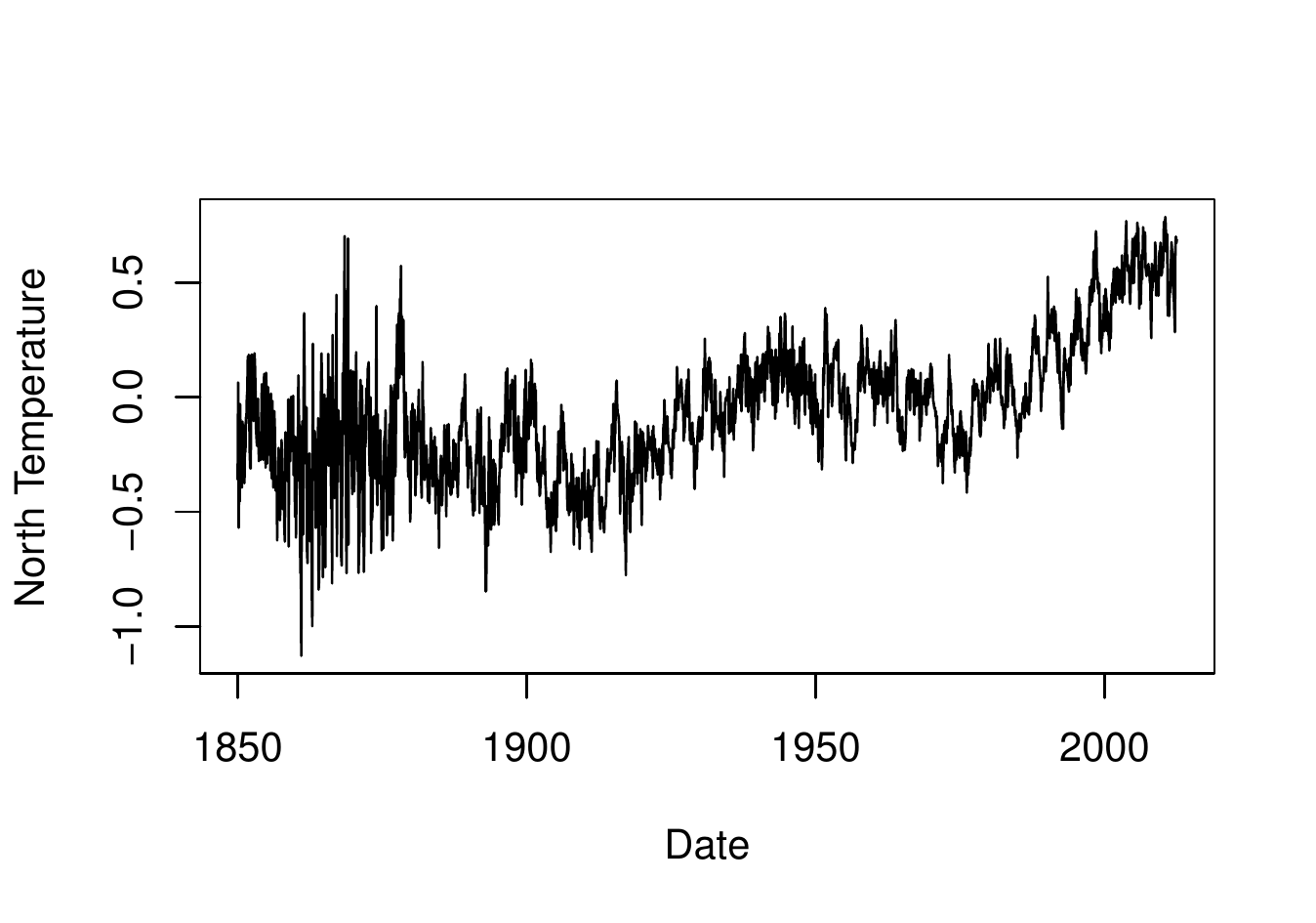}
  \caption{Northern Hemisphere Data}
  \label{fig:nhemdata}
\end{subfigure}%
\begin{subfigure}{.5\textwidth}
  \centering
  \includegraphics[width=7cm,height=7cm]{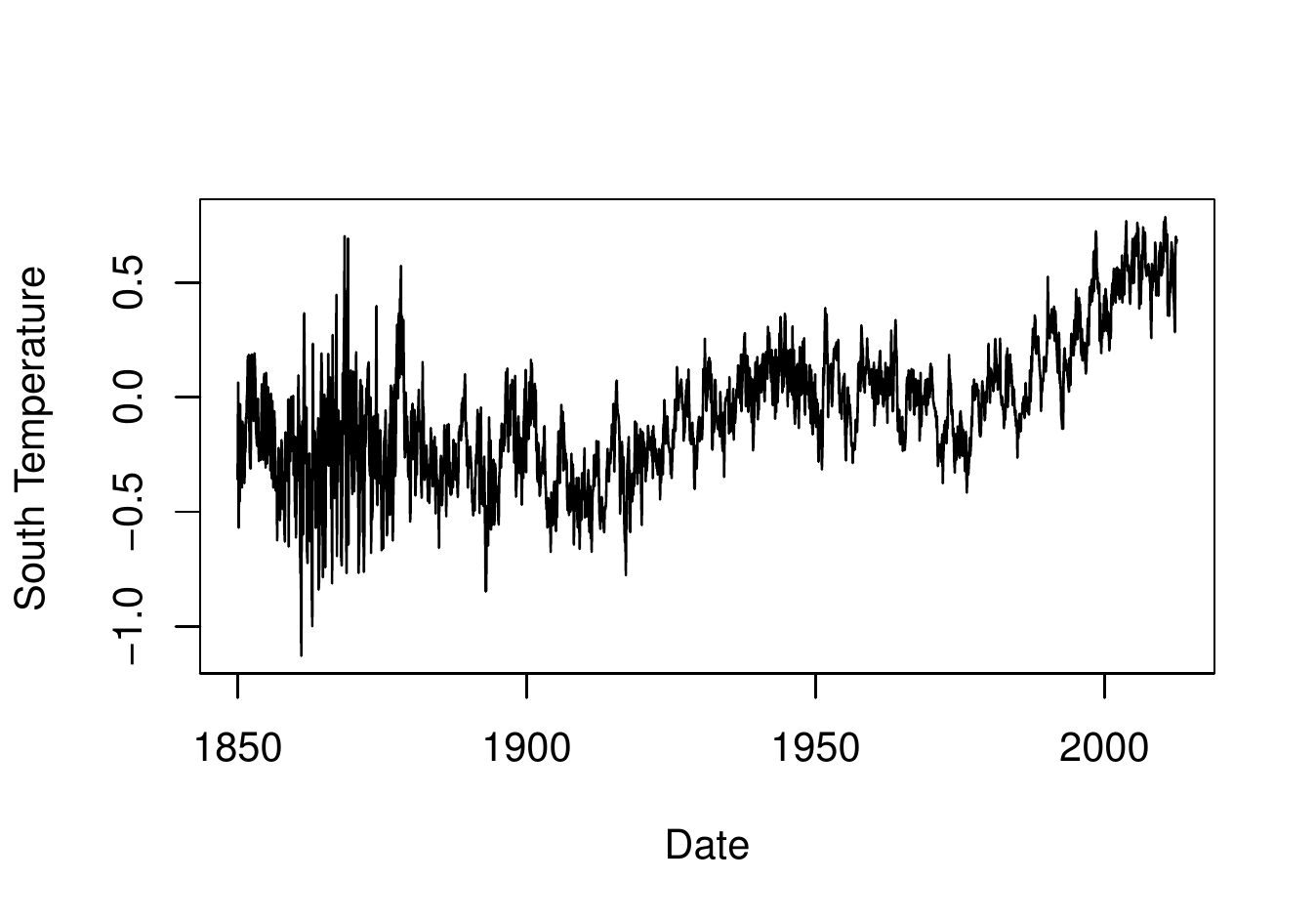}
  \caption{Southern Hemisphere Data}
  \label{fig:shemdata}
\end{subfigure}
\caption{Recorded temperature anomalies for Northern and Southern hemispheres}
\label{fig:realdata}
\end{figure}

Our purpose is to estimate the variance and lag $1$ autocovariance function of these data as a function of time.  For the Northern hemisphere data the approach suggested in our manuscript produces an almost monotonically decaying variance curve that suggests that some nonstationarity is, indeed, present in the data. This monotonic decay is probably due to the increasing number of weather stations recording the data over time. The variance of the Southern hemisphere data is also mostly decreasing although the decay is not as clearly monotonic as for the Northern hemisphere data. Note that, for both sets of data, the lag $1$ autocovariance is very small in magnitude; however, the horizontal zero line added to both autocovariance plots is clearly not fully inside the uniform confidence interval, indicating that the temperature series are not white noises. 

\begin{figure}[t]
\centering
\begin{subfigure}{.5\textwidth}
  \centering
  \includegraphics[width=7cm]{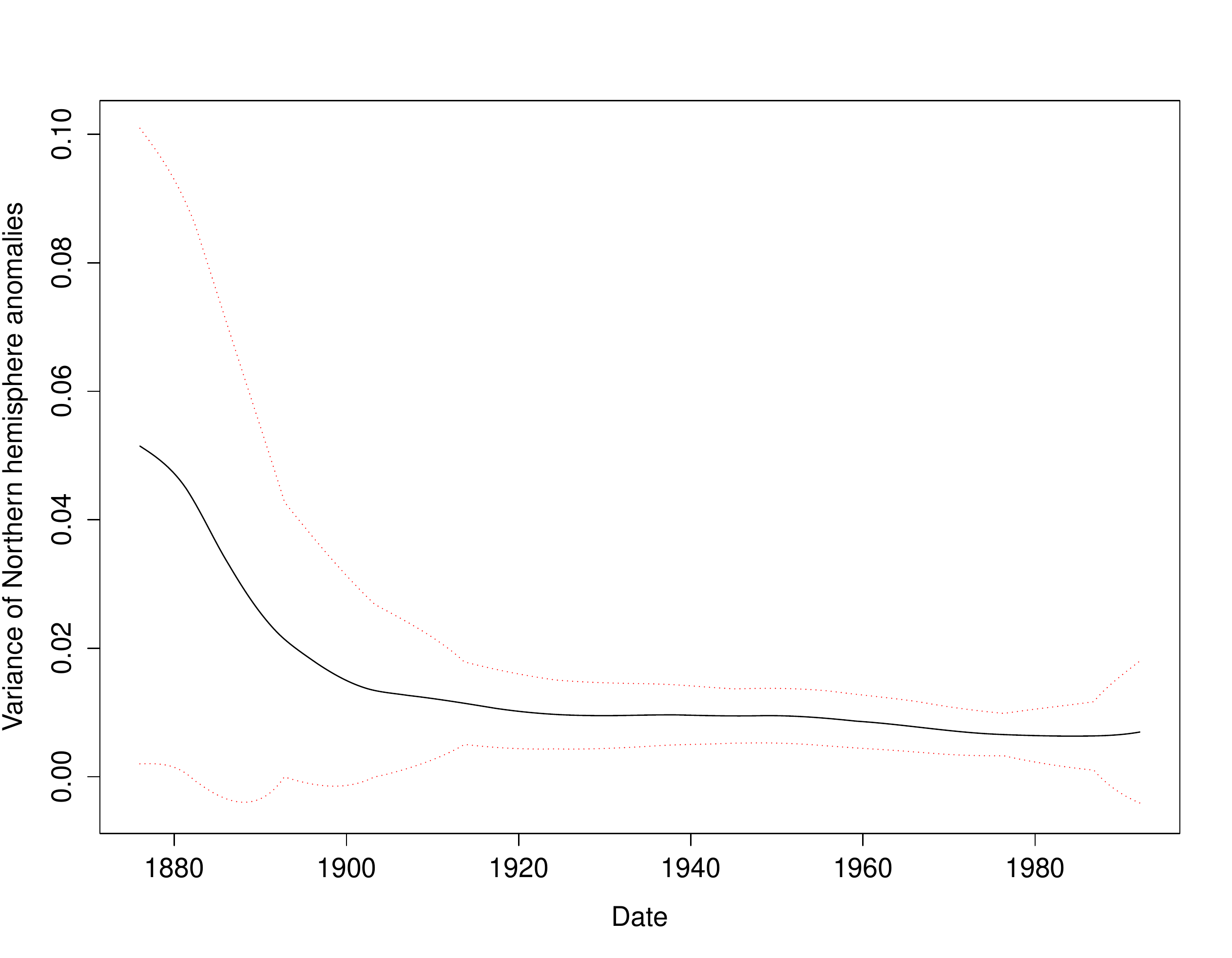}
  \label{fig:realdata1}
\end{subfigure}%
\begin{subfigure}{.5\textwidth}
  \centering
  \includegraphics[width=7cm]{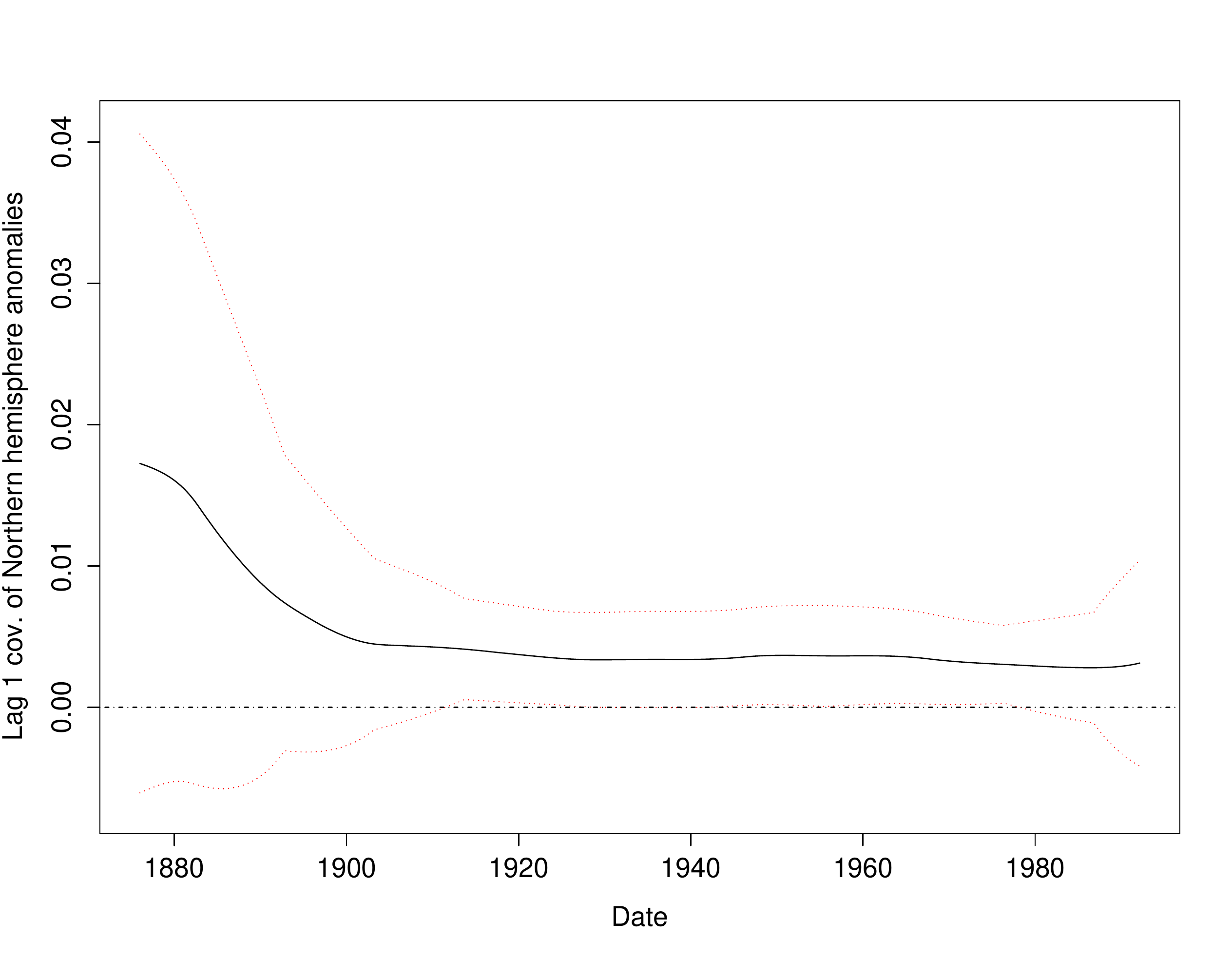}
  \label{fig:realdata2}
\end{subfigure}%
\caption{Variance (left) and lag $1$ autocovariance function (right) of Northern hemisphere anomalies}
\end{figure}

\begin{figure}[h]
\centering
\begin{subfigure}{.5\textwidth}
  \centering
  \includegraphics[width=7cm]{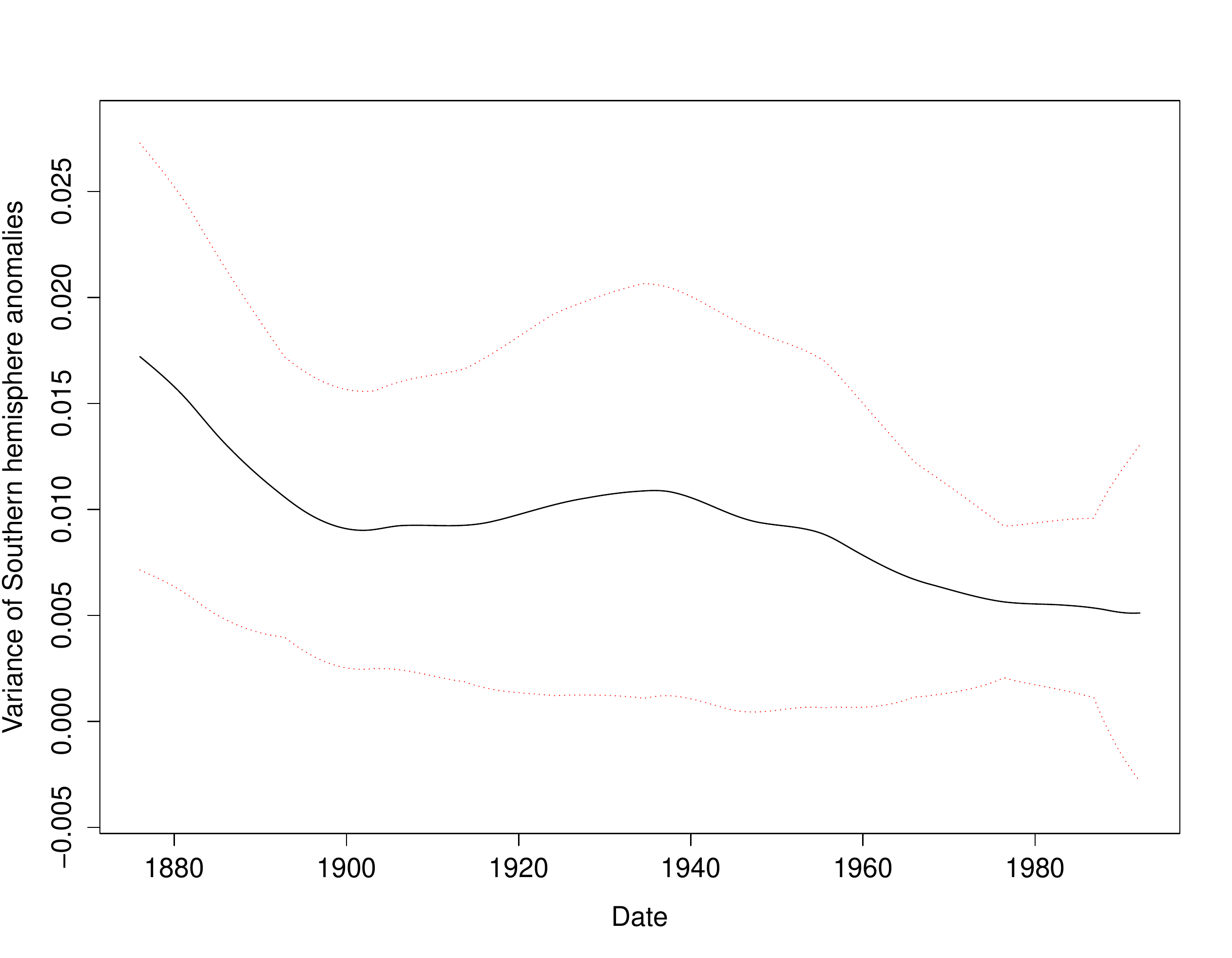}
  \label{fig:realdata11}
\end{subfigure}%
\begin{subfigure}{.5\textwidth}
  \centering
  \includegraphics[width=7cm]{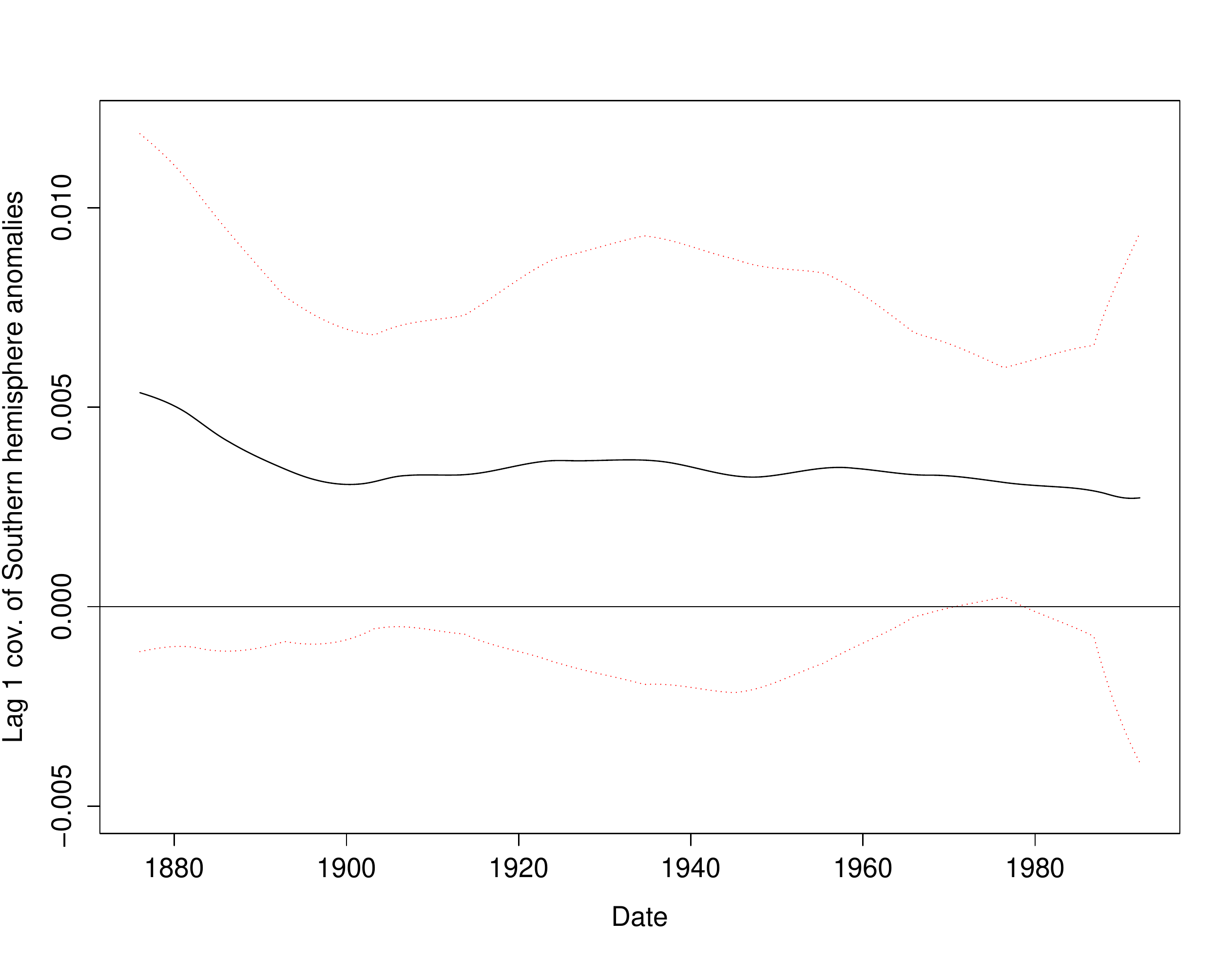}
  \label{fig:realdata22}
\end{subfigure}%
\caption{Variance (left) and lag $1$ autocovariance function (right) of Southern hemisphere anomalies}
\end{figure}

Now, we are interested in testing whether for these data there exist change points in the mean function, namely, $$H_0: \mu_1=\mu_2=\cdots=\mu_n=\mu~\longleftrightarrow~
H_1: \mu_i\neq \mu_j~\text{~for some}~1\le i<j\le n.$$ To this end, we will use the robust bootstrap test for nonstationary time series proposed by \cite{Zhou2013}[Section 4]. For Northern Hemisphere data, the robust bootstrap test yields a $<0.1\%$ $p$-values with 10000 bootstrap samples, which provides a very strong evidence against the null hypothesis of no structural change in mean. On the other hand, we applied the robust bootstrap to the Southern Hemisphere data. The corresponding $p$-value of the test with 10000 bootstrap samples is also $<0.1\%$, which also shows a strong evidence against $H_0.$ As a result, the test further illustrates the usefulness of our method for constructing SCB with finite change points. Over some time periods, the data with wild fluctuations indicates a change in mean and suggests the non-stationarity, as pointed by \cite{rao2004nonstationary}.
 
\vspace{3mm}
{\Large\bf Appendix}
\vspace{2mm}

The following theorem provides the Gaussian approximation result for nonstationary multiple time series, which can be found in Theorem 2 of \cite{ZW2010}.

\begin{theorem}[Theorem 2 in \cite{ZW2010}]\label{thm3}
	Suppose that Assumptions \ref{a1} and \ref{a2} hold. A partial sum process can be defined as $\tilde{S}_i^k=\sum_{j=1}^i\eps_j^k$ for $k=1,...,h$. Then, on a richer probability space, there exist i.i.d. standard normal random variables $u_1,u_2,...,$ and a process $\widehat{S}_i^k$ such that $\{\widetilde{S}_i^k\}_{i=1}^n\overset{\mathcal{D}}{=}\{\widehat{S}_i^k\}_{i=1}^n$ and 
	\begin{equation}\label{e15}
	\max_{i\leq n}\left|\widehat{S}_i^k- \sum_{j=1}^i\sigma_k(t_j)u_j\right|=
	\bigO_{\Pr}\left(n^{2/5}\log(n)\right),
	\end{equation}
	where $\sigma_k(\cdot)$ is defined as \cref{e12}.
\end{theorem}

\cref{thm3} implies the Gaussian approximation for a partial sum of a locally stationary process. Note that, due to the result stated in \cref{lemma2}, the physical dependence measure $\delta_4(H_k,i)$ has different types of the polynomial decay under two circumstances, which is more complicated than that of Corollary 2 in \cite{WZ2011}. But letting the order of the $m$-dependence sequence larger than $k$ and making a careful check of the proof of Corollary 2 in \cite{WZ2011}, we can obtain the same argument. Owing to the non-stationarity, the approximated Gaussian process $\{\sum_{j=1}^i\sigma_k(t_j)u_j\}_{i=1}^n$ has independent but possibly non-identically distributed increments. 

\textbf{Proof of \cref{lemma1}}.
	For $j\in \mathbb{Z}$, define the projection operator $\mathcal{P}_j(\cdot)=\EE(\cdot|\mathcal{F}_j)-\EE(\cdot|\mathcal{F}_{j-1})$, then we can write $x_i=\sum_{j=-\infty}^i\mathcal{P}_j(x_i)$. Denote $t_i=i/n$, we have
	\begin{align*}
	|\gamma_{k}(t_i)|&=\left|\EE\left(\sum_{j=-\infty}^i \mathcal{P}_j(x_i)\sum_{j=-\infty}^{i-k}\mathcal{P}_j(x_{i-k})
	\right)\right|\\
	&\le \left|\EE\left(\sum_{j=-\infty}^\infty\mathcal{P}_j(x_i)
	\mathcal{P}_j(x_{i-k})\right)\right|\\
	&\le \sum_{j=-\infty}^\infty\Vert\mathcal{P}_j(x_i)\Vert_2\cdot
	\Vert\mathcal{P}_j(x_{i-k})\Vert_2\\
	&\le \sum_{j=-\infty}^\infty \delta_2(G,i-j)\delta_2(G,i-k-j). 
	\end{align*} 
	The first inequality follows by the orthogonality of $\mathcal{P}_j(\cdot)$ and the second inequality is due to Fubini's theorem and Cauchy-Schwartz inequality. The last inequality follows from the argument in \cite[Theorem 1]{Wu2005}. Therefore, with \cref{a1}, there exists a constant $C>0$ such that $|\gamma_k(t_i)|\le Ck^{-2}$. This concludes our proof. \qed
	\vskip 0.3cm
	
	\textbf{Proof of \cref{lemma2}}.
	Now, we consider the locally stationary process $(x_l-x_{l-k})^2$ and let $x_l'$ be the coupled process of $x_l$ with $\zeta_0$ replaced by an i.i.d. copy $\zeta_0'$. Then for each $k=1,...,h$,
	\begin{align*}
	\Vert H_k(t,\mathcal{F}_l)-H_k(t,\mathcal{F}_l')\Vert_4
	=&\Vert (x_l-x_{l-k})^2-(x_l'-x_{l-k}')^2\Vert_4\\
	=&\Vert (x_l-x_{l-k}+x_l'-x_{l-k}') (x_l-x_{l-k}-x_l'+x_{l-k}')\Vert_4\\
	\le& \Vert x_l-x_{l-k}+x_l'-x_{l-k}'\Vert_8 \cdot
	\Vert x_l-x_{l-k}-x_l'+x_{l-k}'\Vert_8\\
	\le& 4\sup_{t\in[0,1]}\Vert x_l \Vert_8 \cdot (\sup_{t\in[0,1]}\Vert x_l-x_l'\Vert_8+\sup_{t\in[0,1]}\Vert x_{l-k}-x_{l-k}'\Vert_8)\\
	=& 4\sup_{t\in[0,1]}\Vert x_l\Vert_8
	\begin{cases}
	\delta_8(G,l), &\text{if}~0< l < k,\\
	\delta_8(G,l)+\delta_8(G,l-k+1), &\text{if}~l\ge k.
	\end{cases}\\
	=&\begin{cases}
	\bigO(l^{-2}), &\text{if}~0< l < k,\\
	\bigO(l^{-2})+\bigO((l-k+1)^{-2}), &\text{if}~l\ge k.
	\end{cases}
	\end{align*}
	\qed
	\vskip 0.3cm
	
	\textbf{Proof of \cref{lemma3}}.
		By \cref{a3}, we know that $\beta_k(\cdot)$ is also Lipschitz continuous, thus with \cref{e4},
		\begin{align*}
		&\Vert H_k(t,\mathcal{F}_i)-H_k(s,\mathcal{F}_i)\Vert_2\\
		=&\Vert [G(t,\mathcal{F}_i)-
		G(t-\frac{k}{n},\mathcal{F}_{i-k})]^2-\beta_k(t)-
		[G(s,\mathcal{F}_i)-
		G(s-\frac{k}{n},\mathcal{F}_{i-k})]^2+\beta_k(s)\Vert_2\\
		\le& \Vert\beta_k(t)-\beta_k(s)\Vert_2+\Vert (G(t,\mathcal{F}_i)-
		G(t-\frac{k}{n},\mathcal{F}_{i-k})+
		G(s,\mathcal{F}_i)-
		G(s-\frac{k}{n},\mathcal{F}_{i-k})\Vert_4\\
		{}&{}\cdot \Vert (G(t,\mathcal{F}_i)-
		G(t-\frac{k}{n},\mathcal{F}_{i-k})-
		G(s,\mathcal{F}_i)+
		G(s-\frac{k}{n},\mathcal{F}_{i-k})\Vert_4\\
		\le& C|t-s|+4\sup_{t\in[0,1]}\Vert x_i\Vert_4\cdot
		(\Vert G(t,\mathcal{F}_i)-
		G(s,\mathcal{F}_i)\Vert_4+\Vert
		G(t-\frac{k}{n},\mathcal{F}_{i-k})-
		G(s-\frac{k}{n},\mathcal{F}_{i-k})\Vert_4)\\
		\le&C|t-s|.
		\end{align*}
		The first inequality follows from the triangle inequality and Minkowski's inequality. The second inequality uses elementary calculation and the last line follows by \cref{a2}. On the other hand, \cref{a3} implies that $\beta_k(t)$ is bounded on the compact $[0,1]$. Then, 
		\begin{align*}
		\sup_{t\in[0,1]}\Vert H_k(t,\mathcal{F}_i)\Vert_4&=
		\sup_{t\in[0,1]}\Vert (x_i-x_{i-k})(x_i-x_{i-k})+\beta_k(t_i)\Vert_4\\
		&\le (\sup_{t\in[0,1]}\Vert x_i-x_{i-k}\Vert_8)^2+\sup_{t\in[0,1]} \Vert\beta_k(t_i)\Vert_4\\
		&\le (\sup_{t\in[0,1]}\Vert x_i\Vert_8+
		\sup_{t\in[0,1]}\Vert x_{i-k}\Vert_8)^2
		+\sup_{t\in[0,1]} \Vert\beta_k(t_i)\Vert_4\\
		&<\infty.
		\end{align*} \qed
	
		To prove \cref{thm2}, we need to introduce the following lemmas.
\begin{lemma}\label{l5}
	Let $F_n^k(t)=\sum_{i=1}^nu_iK_b(t_i-t)$, where $u_i,~i \in \mathbb{Z}$ be i.i.d. $N(0,1)$ random variables. Suppose that $b\to 0$ and $nb/\log(n) \to \infty,~m^\ast=1/b$. Then, 
	\begin{align*}
	\lim_{n\to\infty}\left(\Pr\left[\frac{1}{\sqrt{nb\phi_0}}\sup_{t\in\mathcal{T}}|F_n^k(t)|-B_K(m^\ast)\leq \frac{u}{\sqrt{2\log(m^\ast)}}\right]\right)\\
	=\exp\{-2\exp(-u)\}.
	\end{align*}
\end{lemma} 
\begin{proof}
	Similar to Lemma 2 in \cite{ZW2010}.
\end{proof} 

\begin{lemma}\label{l6}
	Let $D_k(t)=\frac{1}{nb}\sum_{i=1}^n\eps_i^kK_b(t_i-t)$. Assume that $\sigma_k(t)$ is Lipschitz continuous and bounded away from 0 on $[0,1]$ and $\log(n)/{n^{3/5}b}+b\log(n) \to 0$. Then we have
	\begin{align*}
	\lim_{n\to\infty}\left(\Pr\left[\sqrt{\frac{nb}{\phi_0}}\sup_{t\in\mathcal{T}}\left|\sigma_k^{-1}(t)D_k(t)\right|-B_K(m^\ast)\leq \frac{u}{\sqrt{2\log(m^\ast)}}\right]\right)\\
	=\exp\{-2\exp(-u)\}.
	\end{align*}
\end{lemma}
\begin{proof} By summation by part and the result of \cref{thm3}, we have 
	\begin{align*}
	\sup_{t\in\mathcal{T}}|D_k(t)-\Xi_k(t)|& \le
	\left|\frac{1}{nb}\sum_{i=1}^n\left[\sum_{j=1}^i[\eps_i^j-\sigma_k(t_j)u_j]K\left(\frac{t_i-t}{b}\right)\right]\right|
	=\bigO_{\Pr}\left(\frac{\log(n)}{n^{3/5}b}\right), 
	\end{align*}
	where $\Xi_k(t)=\frac{1}{nb}\sum_{i=1}^n\sigma_k(t_i)u_iK_{b}(t_i-t)$. Since $\sigma_k(\cdot)$ is Lipschitz continuous, then
	\begin{equation}\label{e14}
	\sup_{t\in \mathcal{T}}\left|\Xi_k(t)-\frac{1}{nb} \sigma_k(t)\sum_{i=1}^nu_i K\left(\frac{t_i-t}{b}\right) \right|=\bigO_{\Pr}\left(\sqrt{\frac{b\log(n)}{n}}\right).
	\end{equation} 
	Hence, \cref{l6} holds following the above equations and the arguments of \cref{l5}.
\end{proof}
\vskip 0.3cm
\textbf{Proof of \cref{thm2}}.
	Recall \cref{e9,e10,e11} and denote $\eta_k(t)=(\beta_k(t),b\beta_k'(t))^\top
	=[Q_n^k(t)]^{-1}R_n^k(t)$, 
	we have
	\begin{equation}\label{e13}
	Q_n^k(t)(\widehat{\eta}_k(t)-\eta_k(t))=\left(
	\begin{array}{c}
	~b^2Q_{n,2}^k(t)\{\beta_k''(t)/2+\bigO(b)\}\\
	b^2Q_{n,3}^k(t)\{\beta_k''(t)/2+\bigO(b)\}\\
	\end{array}\right)
	+\Lambda_n^k(t)+\Theta_n^k(t)+T_n^k(t),
	\end{equation}
	where $T_n^k(t)=(T_{n,0}^k(t),T_{n,1}^k(t))^\top,
	~\Lambda_n^k(t)=(\Lambda_{n,0}^k(t),\Lambda_{n,1}^k(t))^\top,~
	\Theta_n^k(t)=(\Theta_{n,0}^k(t),\Theta_{n,1}^k(t))^\top$ and
	\begin{align*}
	T_{n,l}^k(t)&=\frac{1}{nb}\sum_{i=1}^n\eps_i^k\left(\frac{t_i-t}{b}\right)^lK\left(\frac{t_i-t}{b}\right),\\
	\Lambda_{n,l}^k(t)&=\frac{1}{nb}\sum_{i=1}^n\lambda_i^k\left(\frac{t_i-t}{b}\right)^lK\left(\frac{t_i-t}{b}\right),\\
	\Theta_{n,l}^k(t)&=\frac{1}{nb}\sum_{i=1}^n\theta_i^k\left(\frac{t_i-t}{b}\right)^lK\left(\frac{t_i-t}{b}\right), ~l=0,1.
	\end{align*}
	As a consequence of the weak law of large numbers, we know that
	$Q_n^k(t) \xrightarrow{\mathcal{P}} Q(t)$ in the sense that each element converges in probability,
	where
	$$Q(t)=\left(
	\begin{array}{cc}
	1 & 0\\
	0 & \mu_2
	\end{array}
	\right).$$\label{eq1}
	Hence, from \cref{e13} and the results of \cref{thm1}, we have 
	\begin{align*}
	&\sqrt{nb}\sup_{t\in \mathcal{T}}\left|\widehat{\beta}_{k,b}(t)-\beta_k(t)-
	T_{n,0}^k(t)\right|\\=&
	\bigO\left(n^{1/2}b^{7/2}\right)+\bigO_{\Pr}
	\left(\sqrt{\frac{1}{n^{1-2\alpha-4\beta}b}}\right)
	+\bigO_\Pr\left(\sqrt{\frac{\log(n)}{n^{1-2\alpha-2\beta}b}}\right)
	+\bigO(\sqrt{nb^5\log(n)}) \xrightarrow{\mathcal{P}} 0.
	\end{align*} 
	Furthermore by the Proposition 6 in \cite{Zhou2009}, we have for any $k=1,...,h$,
	$$(nb)^{1/2}T_n^k(t)\Rightarrow N(0,\nu_k(t)),~~
	\nu_k(t)=\sigma_k^2(t)~{\rm diag} (\phi_0,\phi_2).$$ 
	Next, we treat $T_{n,0}^k(t)$ as $D_k(t)$ in \cref{l6}, then \cref{thm2} follows. \qed
	\vskip 0.5cm
	
	\textbf{Proof of \cref{thm1}}.
		Recall that our model contains $d=\bigO(n^\alpha)$ change points with the maximal size $\Delta_n=\bigO(n^\beta)$ on $\mu(\cdot)$ and let $\Omega$ and $\Lambda_k$ be the $n$-dimensional vectors with the entrywises $\omega_n^b(t,i)$ and $(\mu_i-\mu_{i-k})^2$ for $i=1,...,n$, respectively. Due to the fact that $\mu(\cdot)$ is Lipschitz continuous, one can see that $\Lambda_k$ consists of $kd$ components being $\bigO(n^{2\beta})$ and other components being $\bigO(1/n^2)$. For each $k=1,...,h$ and any fixed $t\in[0,1]$, by \cref{e3,e4,e8} we have
		\begin{align*}
		\sup_{t\in\mathcal{T}}\left|\widehat{\beta}_{k,b}(t)
		-\beta_k(t)\right|
		=& \sup_{t\in\mathcal{T}}\left|\sum_{i=1}^n\omega_n^b(t,i)
		(\beta_k(t_i)+\eps_i^k+\lambda_i^k+\theta_i^k)-
		\sum_{i=1}^n\omega_n^b(t,i)\beta_k(t)\right| \\
		=& \sup_{t\in\mathcal{T}}\left|
		\sum_{i=1}^n\omega_n^b(t,i)\beta_k'(t_i)(t-t_i)+
		\sum_{i=1}^n\omega_n^b(t,i)\left[\frac{\beta_k''(t_i)}{2}+
		\bigO(b^2)\right](t-t_i)^2\right. \notag\\
		&{}+\left.\sum_{i=1}^n\omega_n^b(t,i)\eps_i^k
		+\sum_{i=1}^n\omega_n^b(t,i)\lambda_i^k+
		\sum_{i=1}^n\omega_n^b(t,i)\theta_i^k\right| \\
		\le &\sup_{t\in\mathcal{T}}\left|\sum_{i=1}^n
		\omega_n^b(t,i)\left[\frac{\beta_k''(t_i)}{2}+
		\bigO(b^2)\right](t-t_i)^2\right|+\sup_{t\in\mathcal{T}}
		\left|\sum_{i=1}^n\omega_n^b(t,i)\eps_i^k\right| \\
		&{}+\sup_{t\in\mathcal{T}}\left|\sum_{i=1}^n
		\omega_n^b(t,i)\lambda_i^k\right|+\sup_{t\in\mathcal{T}}
		\left|\sum_{i=1}^n\omega_n^b(t,i)\theta_i^k\right|:=I+II+III+IV.
		\end{align*}
		It is obvious that $I=\bigO_{\Pr}(b^2)$. Then we will apply a chaining argument for calculating $II$. For $t\in[0,1]$, define the sampling time points as $t=s/n,~s=0,1,...,n$ and let $\Pi_t=\sum_{i=1}^n\omega_n^b(t,i)\epsilon_i^k$. Then for each time point $s/n$, we have
		\begin{align*}
		\max_{0\le s\le n}\left|\Pi_{\frac{s}{n}}\right|&=
		\max_{0\le s\le n}\left|\sum_{i=1}^n
		\omega_n^b\left(\frac{s}{n},i\right)\eps_i^k\right|\\
		&\le \max_{0\le s\le n}\left|\sum_{i=1}^n
		\omega_n^b\left(\frac{s}{n},i\right)\sigma_k
		\left(\frac{s}{n}\right)u_i\right|
		+\bigO_{\Pr}\left(\frac{n^{2/5}\log n}{nb}\right)\\
		&=\bigO_{\Pr}\left(\sqrt{\frac{\log n}{nb}}\right)+
		\bigO_{\Pr}\left(\frac{\log n}{n^{3/5}b}\right) .
		\end{align*}
		The first inequality uses \cref{thm3} and the triangle inequality. Next, we consider the difference between $\Pi_t$ and $\Pi_{\frac{s}{n}}$.
		\begin{align*}
		\sup_{t\in [s/n,(s+1)/n]}
		\left|\Pi_t-\Pi_{\frac{s}{n}}\right|&=
		\sup_{t\in[s/n,(s+1)/n]}\left|
		\sum_{i=1}^n\left[\omega_n^b(t,i)-
		\omega_n^b\left(\frac{s}{n},i\right)\right]
		\sigma_k\left(\frac{s}{n}\right)\eps_i^k\right|\\
		&=\sup_{t\in[s/n,(s+1)/n]}\int_{\frac{s}{n}}
		^{t}\left|\sum_{i=1}^n
		{\omega_n^b}'(v,i)\sigma_k\left(\frac{s}{n}\right)
		\epsilon_i^k\right| \dee v\\
		&\le \int_{\frac{s}{n}}^{\frac{s+1}{n}}\left|\sum_{i=1}^n
		\omega_n^b{'(v,i)}\sigma_k\left(\frac{s}{n}\right)
		\epsilon_i^k\right| \dee v\\
		&=\left|\sum_{i=1}^n\left|\omega_n^b\left(\frac{s+1}{n},i\right)-
		\omega_n^b\left(\frac{s}{n},i\right)\right|
		\sigma_k\left(\frac{s}{n}\right)\epsilon_i^k\right|\\
		&=\bigO_{\Pr}\left(\sqrt{\frac{\log n}{n}}\right)+
		\bigO_{\Pr}\left(\frac{\log n}{n^{3/5}}\right).
		\end{align*}
		Thus, we have
		\begin{align*}
		II=\sup_{t\in\mathcal{T}}\left|\Pi_t\right|&=
		\sup_{t\in\mathcal{T}}\left|\Pi_t-\Pi_{\frac{s}{n}}+
		\Pi_{\frac{s}{n}}\right|\\
		&\le \max_{0\le s\le n}\sup_{t\in [s/n,(s+1)/n]}
		\left|\Pi_t-\Pi_{\frac{s}{n}}\right|+\max_{0\le s\le n}\left|\Pi_{\frac{s}{n}}\right|\\
		&=\bigO_{\Pr}\left(\sqrt{\frac{\log n}{nb}}\right)+
		\bigO_{\Pr}\left(\frac{\log n}{n^{3/5}b}\right).
		\end{align*}
		As for $III$, notice that
		\begin{equation} \label{eq3}
		\left|\sum_{i=1}^n\omega_n^b(t,i)\lambda_i^k\right|= \left|\langle\Omega, \Lambda_k\rangle\right|
		\le \Vert\Omega\Vert_2 \Vert\Lambda_k\Vert_2=
		\bigO_{\Pr}\left(\sqrt{\frac{1}{n^{1-\alpha-4\beta}b}}\right).
		\end{equation}
		In the end, by the similar chaining argument as those in the proof of $I$, we have $$IV=\bigO_{\Pr}\left(\sqrt{\frac{\log n}{n^{1-2\beta}b}}\right)+
		\bigO_{\Pr}\left(\frac{\log n}{n^{3/5-\beta}b}\right).$$ Note that the assumption $\alpha+2\beta<2/5$ entails $\beta<1/5$, therefore by elementary calculation, the above four kinds of bounds all converge to 0 as $n\to \infty$. \qed
\vskip 0.5cm

\textbf{Proof of \cref{thm4}}.
Similar to the proof of Theorem 4 in \cite{ZW2010}. \qed
\vskip 0.5cm

\textbf{Proof of \cref{thm5}}.	
Let $\mathcal{I}'$ be a closed interval in $(0,1)$ such that $\mathcal{I}\subset\mathcal{I}'$ and the two intervals do not share common end points. Recall $\widetilde{\beta}_k(t)=(\beta_h(t),\beta_k(t))^\top$ and denote $\widehat{\beta}_{k,b_k}(t)=(\widehat{\beta}_{h,b_k}(t),\widehat{\beta}_{k,b_k}(t))^\top$. According to \cref{thm1}, it follows that
\begin{equation}\label{e16}
\sup_{i/n\in \mathcal{I}'}\left|\hat{\eps}_i-\tilde{\eps}_i\right|=
\sup_{i/n\in \mathcal{I}'}\left|\widetilde{\beta}_k(t_i)-\widehat{\beta}_{k,b_k}(t_i)\right|=
\bigO_{\Pr}(\chi_n).
\end{equation}
Note that $Q_i^k/(2m+1)$ is the Nadaraya-Waston smoother of the series $\{\tilde{\eps}_i\}$ at $i$ with the rectangle kernel and bandwidth $m/n$. Therefore, for each $k=1,...,h$, we have 
\begin{equation}\label{e17}
\sup_{i/n \in \mathcal{I}'}|Q_i^k|=\bigO_{\Pr}(\sqrt{m\log n}).
\end{equation}
Let $\widehat{Q}_i^k=\sum_{j=-m}^{m}\hat{\eps}_{i+j}$ and $\widehat{N}_i^k=\widehat{Q}_i^k\widehat{Q}_i^{k\top}/(2m+1)$. Then
$$(2m+1)(N_i^k-\widehat{N}_i^k)=(Q_i^k-\widehat{Q}_i^k)
(Q_i^k)^\top+\widehat{Q}_i^k(Q_i^k-\widehat{Q}_i^k)^\top.$$
Substituting equations (\ref{e16}) and (\ref{e17}) into the above equation, we have $\sup_{i/n\in \mathcal{I}'}|N_i^k
-\widehat{N}_i^k|=\bigO_{\Pr}(\nu_n)$ with the assumption $\nu_n\to 0$. By the definitions of $\widetilde{\Sigma}_k(t)$ and $\widehat{\Sigma}_k(t)$, we obtain 
$$\sup_{i/n\in \mathcal{I}'}\left|\widetilde{\Sigma}_k(t)-
\widehat{\Sigma}_k(t)\right|=\bigO_{\Pr}(\nu_n).$$ 
Together with the results of \cref{thm4}, \cref{thm5} holds. \qed

\vskip 0.5cm

\textbf{Proof of \cref{prop1}}.
This proposition follows by \cref{thm2} and \cref{e14} from the proof of \cref{l6}. \qed

\bibliography{second}

\begin{thebibliography}{10}

\bibitem{Brockwell2016}
P.~J. Brockwell and R.~A. Davis.
\newblock {\em Introduction to Time Series and Forecasting}.
\newblock Springer, 2016.

\bibitem{brown2007variance}
L.~D. Brown and M.~Levine.
\newblock Variance estimation in nonparametric regression via the difference
  sequence method.
\newblock {\em The Annals of Statistics}, 35(5):2219--2232, 2007.

\bibitem{cai2009variance}
T.~T. Cai, M.~Levine, and L.~Wang.
\newblock Variance function estimation in multivariate nonparametric regression
  with fixed design.
\newblock {\em Journal of Multivariate Analysis}, 100(1):126--136, 2009.

\bibitem{Chakar.etal.16}
S.~Chakar, E.~Lebarbier, C.~L{\'e}vy-Leduc, and S.~Robin.
\newblock A robust approach for estimating change-points in the mean of an
  {AR}(1) process.
\newblock {\em Bernoulli}, 23(2):1408--1447, 2017.

\bibitem{Craven1978}
P.~Craven and G.~Wahba.
\newblock Smoothing noisy data with spline functions.
\newblock {\em Numerische mathematik}, 31(4):377--403, 1978.

\bibitem{Dahlhaus1997}
R.~Dahlhaus.
\newblock Fitting time series models to nonstationary processes.
\newblock {\em The Annals of Statistics}, 25:1--37, 1997.

\bibitem{dai2017choice}
W.~Dai, T.~Tong, and L.~Zhu.
\newblock On the choice of difference sequence in a unified framework for
  variance estimation in nonparametric regression.
\newblock {\em Statistical Science}, 32(3):455--468, 2017.

\bibitem{dette2011optimal}
H.~Dette, P.~Preu{\ss}, and M.~Vetter.
\newblock A measure of stationarity in locally stationary processes with
  applications to testing.
\newblock {\em Journal of the {A}merican {S}tatistical {A}ssociation},
  106:1113--1124, 2010.

\bibitem{Dette2019}
H.~Dette, W.~Wu, and Z.~Zhou.
\newblock Change point analysis of correlation in non-stationary time series.
\newblock {\em Statistica Sinica}, 29(2):611--643, 2019.

\bibitem{Politis1999}
J.~P.~R. Dimitris N.~Politis and M.~Wolf.
\newblock {\em Subsampling}.
\newblock New York: Springer, 1999.

\bibitem{Ding2019}
X.~Ding and Z.~Zhou.
\newblock Estimation and inference for precision matrices of non-stationary
  time series.
\newblock {\em The Annals of Statistics}, 2019.

\bibitem{Dwivedi2011}
Y.~Dwivedi and S.~S. Rao.
\newblock A test for second-order stationarity of a time series based on the
  discrete fourier transform.
\newblock {\em Journal of Time Series Analysis}, 32:68--91, 2011.

\bibitem{eubank1990testing}
R.~L. Eubank and C.~H. Spiegelman.
\newblock Testing the goodness of fit of a linear model via nonparametric
  regression techniques.
\newblock {\em Journal of the American Statistical Association},
  85(410):387--392, 1990.

\bibitem{fan1994gijbels}
J.~Fan and I.~Gijbels.
\newblock Local polynomial modelling and its applications.
\newblock {\em London: Chapman and Hall}, 1994.

\bibitem{gasser1991flexible}
T.~Gasser, A.~Kneip, and W.~K{\"o}hler.
\newblock A flexible and fast method for automatic smoothing.
\newblock {\em Journal of the {A}merican {S}tatistical {A}ssociation},
  86(415):643--652, 1991.

\bibitem{hall2003using}
P.~Hall and I.~V. Keilegom.
\newblock Using difference-based methods for inference in nonparametric
  regression with time series errors.
\newblock {\em Journal of the Royal Statistical Society: Series B (Statistical
  Methodology)}, 65(2):443--456, 2003.

\bibitem{Hamilton1994}
J.~D. Hamilton.
\newblock {\em Time Series Analysis}.
\newblock Princeton University Press, 1994.

\bibitem{hardle1997local}
W.~H{\"a}rdle and A.~Tsybakov.
\newblock Local polynomial estimators of the volatility function in
  nonparametric autoregression.
\newblock {\em Journal of Econometrics}, 81(1):223--242, 1997.

\bibitem{herrmann1992choice}
E.~Herrmann, T.~Gasser, and A.~Kneip.
\newblock Choice of bandwidth for kernel regression when residuals are
  correlated.
\newblock {\em Biometrika}, 79(4):783--795, 1992.

\bibitem{muller1987estimation}
H.-G. M{\"u}ller and U.~Stadtm{\"u}ller.
\newblock Estimation of heteroscedasticity in regression analysis.
\newblock {\em The Annals of Statistics}, 15(2):610--625, 1987.

\bibitem{hans1988detecting}
H.-G. M{\"u}ller and U.~Stadtm{\"u}ller.
\newblock Detecting dependencies in smooth regression models.
\newblock {\em Biometrika}, 75(4):639--650, 1988.

\bibitem{muller1993variance}
H.-G. M{\"u}ller and U.~Stadtm{\"u}ller.
\newblock On variance function estimation with quadratic forms.
\newblock {\em Journal of Statistical Planning and Inference}, 35(2):213--231,
  1993.

\bibitem{Nason2013}
G.~Nason.
\newblock A test for second-order stationarity and approximate confidence
  intervals for localized autocovariances for locally stationary time series.
\newblock {\em Journal of the Royal Statistical Society. Series B (Statistical
  Methodology)}, 75:879--904, 2013.

\bibitem{Nason2000}
G.~P. Nason, R.~von Sachs, and G.~Kroisandt.
\newblock Wavelet processes and adaptive estimation of the evolutionary wavelet
  spectrum.
\newblock {\em Journal of the Royal Statistical Society. Series B. Statistical
  Methodology}, 62:271--292, 2000.

\bibitem{Papatoditis2010}
E.~Paparoditis.
\newblock Validating stationarity assumptions in time series analysis by
  rolling local periodograms.
\newblock {\em Journal of the American Statistical Association},
  105(490):839--851, 2010.

\bibitem{park2006simple}
B.~U. Park, Y.~K. Lee, T.~Y. Kim, and C.~Park.
\newblock A simple estimator of error correlation in non-parametric regression
  models.
\newblock {\em Scandinavian Journal of Statistics}, 33(3):451--462, 2006.

\bibitem{rao2004multiple}
S.~S. Rao.
\newblock On multiple regression models with nonstationary correlated errors.
\newblock {\em Biometrika}, 91(3):645--659, 2004.

\bibitem{rao2004nonstationary}
T.~S. Rao and E.~Tsolaki.
\newblock Nonstationary time series analysis of monthly global temperature
  anomalies.
\newblock In {\em Time Series Analysis and Applications to Geophysical
  Systems}, pages 73--103. Springer, 2004.

\bibitem{rice1984bandwidth}
J.~Rice et~al.
\newblock Bandwidth choice for nonparametric regression.
\newblock {\em The Annals of Statistics}, 12(4):1215--1230, 1984.

\bibitem{tecuapetla2017autocovariance}
I.~Tecuapetla-G{\'o}mez and A.~Munk.
\newblock Autocovariance estimation in regression with a discontinuous signal
  and $m$-dependent errors: A difference-based approach.
\newblock {\em Scandinavian Journal of Statistics}, 44(2):346--368, 2017.

\bibitem{wang2008effect}
L.~Wang, L.~D. Brown, T.~T. Cai, and M.~Levine.
\newblock Effect of mean on variance function estimation in nonparametric
  regression.
\newblock {\em The Annals of Statistics}, 36(2):646--664, 2008.

\bibitem{Wu2005}
W.~B. Wu.
\newblock Nonlinear system theory: another look at dependence.
\newblock {\em Proceedings of the National Academy of Sciences of the United
  States of America}, 102(40):14150--14154, 2005.

\bibitem{WZ2011}
W.~B. Wu and Z.~Zhou.
\newblock Gaussian approximations for non-stationary multiple time series.
\newblock {\em Statistica Sinica}, 21(3):1397--1413, 2011.

\bibitem{zhou2015optimal}
Y.~Zhou, Y.~Cheng, L.~Wang, and T.~Tong.
\newblock Optimal difference-based variance estimation in heteroscedastic
  nonparametric regression.
\newblock {\em Statistica Sinica}, 25:1377--1397, 2015.

\bibitem{Zhou2010}
Z.~Zhou.
\newblock Nonparametric inference of quantile curves for nonstationary time
  series.
\newblock {\em The Annals of Statistics}, 38(4):2187--2217, 2010.

\bibitem{Zhou2013}
Z.~Zhou.
\newblock Heteroscedasticity and autocorrelation robust structural change
  detection.
\newblock {\em Journal of the American Statistical Association}, 108:726--740,
  2013.

\bibitem{Zhou2009}
Z.~Zhou and W.~B. Wu.
\newblock Local linear quantile estimation for nonstationary time series.
\newblock {\em The Annals of Statistics}, 37(5B):2696--2729, 2009.

\bibitem{ZW2010}
Z.~Zhou and W.~B. Wu.
\newblock Simultaneous inference of linear models with time varying
  coefficients.
\newblock {\em Journal of the Royal Statistical Society. Series B (Statistical
  Methodology).}, 72(4):513--531, 2010.

\end{thebibliography}

\end{document}